\documentclass[10pt]{article}
\usepackage[T1]{fontenc}
\usepackage[latin9]{inputenc}
\usepackage[paperwidth=19cm,paperheight=28cm]{geometry}
\usepackage{color}
\usepackage{amsthm}
\usepackage{amsmath}
\usepackage{amssymb}
\usepackage{esint}
\usepackage[numbers]{natbib}
\usepackage[unicode=true,pdfusetitle,
 bookmarks=true,bookmarksnumbered=false,bookmarksopen=false,
 breaklinks=false,pdfborder={0 0 0},backref=false,colorlinks=true]
 {hyperref}
\hypersetup{
 citecolor=blue,urlcolor=blue}

\makeatletter
\usepackage{enumitem}		
\newcommand{\lyxaddress}[1]{
\par {\raggedright #1
\vspace{1.4em}
\noindent\par}
}
\theoremstyle{plain}
\newtheorem{thm}{\protect\theoremname}
  \theoremstyle{remark}
  \newtheorem{rem}[thm]{\protect\remarkname}
  \theoremstyle{definition}
  \newtheorem{condition}[thm]{\protect\conditionname}
  \theoremstyle{plain}
  \newtheorem{lem}[thm]{\protect\lemmaname}
  \theoremstyle{plain}
  \newtheorem{prop}[thm]{\protect\propositionname}
  \theoremstyle{plain}
  \newtheorem{cor}[thm]{\protect\corollaryname}
  \theoremstyle{definition}
  \newtheorem*{example*}{\protect\examplename}
  \theoremstyle{plain}
  \newtheorem{conjecture}[thm]{\protect\conjecturename}
  \theoremstyle{remark}
  \newtheorem*{acknowledgement*}{\protect\acknowledgementname}

\RequirePackage[OT1]{fontenc}
\RequirePackage{amsthm,amsmath,natbib}

\newcounter{hypA}
\renewenvironment{condition}{\refstepcounter{hypA}\begin{itemize}
\item[({\bf A\arabic{hypA}})]}{\end{itemize}}

\makeatother

  \providecommand{\acknowledgementname}{Acknowledgement}
  \providecommand{\conditionname}{Condition}
  \providecommand{\conjecturename}{Conjecture}
  \providecommand{\corollaryname}{Corollary}
  \providecommand{\examplename}{Example}
  \providecommand{\lemmaname}{Lemma}
  \providecommand{\propositionname}{Proposition}
  \providecommand{\remarkname}{Remark}
\providecommand{\theoremname}{Theorem}

\begin{document}

\title{Uniform Ergodicity of the Iterated Conditional SMC and Geometric
Ergodicity of Particle Gibbs samplers}

\author{Christophe Andrieu, Anthony Lee and Matti Vihola}

\maketitle

\lyxaddress{School of Mathematics, University of Bristol,}

\lyxaddress{Department of Statistics, University of Warwick,}

\lyxaddress{Department of Statistics, University of Oxford.}

\begin{abstract}
We establish quantitative bounds for rates of convergence and asymptotic
variances for iterated conditional sequential Monte Carlo (i-cSMC)
Markov chains and associated particle Gibbs samplers~\citep{andrieu-doucet-holenstein}.
Our main findings are that the essential boundedness of potential
functions associated with the i-cSMC algorithm provide necessary and
sufficient conditions for the uniform ergodicity of the i-cSMC Markov
chain, as well as quantitative bounds on its (uniformly geometric)
rate of convergence. Furthermore, we show that the i-cSMC Markov chain
cannot even be geometrically ergodic if this essential boundedness
does not hold in many applications of interest. Our sufficiency and
quantitative bounds rely on a novel non-asymptotic analysis of the
expectation of a standard normalizing constant estimate with respect
to a ``doubly conditional'' SMC algorithm. In addition, our results
for i-cSMC imply that the rate of convergence can be improved arbitrarily
by increasing $N$, the number of particles in the algorithm, and
that in the presence of mixing assumptions, the rate of convergence
can be kept constant by increasing $N$ linearly with the time horizon.
 We translate the sufficiency of the boundedness condition for i-cSMC
into sufficient conditions for the particle Gibbs Markov chain to
be geometrically ergodic and quantitative bounds on its geometric
rate of convergence, which imply convergence of properties of the
particle Gibbs Markov chain to those of its corresponding Gibbs sampler.
These results complement recently discovered, and related, conditions
for the particle marginal Metropolis--Hastings (PMMH) Markov chain.\medskip{}

\end{abstract}

\section{Introduction\label{sec:Introduction}}

Particle Markov chain Monte Carlo (P-MCMC) methods are a set of recently
proposed sampling techniques particularly well suited to the Bayesian
estimation of static parameters in general state-space models~\citep{andrieu-doucet-holenstein},
although their scope extends beyond this class of models. At an abstract
level, once the likelihood function and prior are defined, inference
for this class of models relies on a probability distribution $\pi\big({\rm d}\theta\times{\rm d}x\big)$,
defined on some measurable space $\left(\Theta\times\mathsf{X},\mathcal{B}(\Theta)\times\mathcal{B}\mathsf{(X)}\right)$,
where $\theta$ is generally a low dimensional static parameter, the
static parameter, while $x$, the hidden state of the system, is a
large vector with a non-trivial dependence structure. Here, $\mathcal{B}(\,\cdot\,)$
denotes the $\sigma$-algebra related to the corresponding space.
In practice the complexity of such probability distributions requires
the use of sampling techniques to effectively carry out inference.
When $\theta$ is known sequential Monte Carlo methods (SMC), or particle
filters, are particularly suitable to carry out inference about $x$
by approximately sampling from the conditional distribution $\pi_{\theta}\big({\rm d}x\big)$.
These algorithms rely on interacting particle systems and their performance
and accuracy can be improved by increasing the number $N$ of such
particles. P-MCMC realises the synthesis between SMC methods and classical
Markov chain Monte Carlo (MCMC) methods, that is it allows the construction
of Markov transition probabilities leaving $\pi\big({\rm d}\theta\times{\rm d}x\big)$
at least marginally invariant and from which it is possible to sample
realisations $\{(\theta_{i},X_{i}),i\geq0\}$ with attractive efficiency
properties.

The particle marginal Metropolis--Hastings (PMMH) method is one such
algorithm, which takes advantage of the availability of unbiased estimators
of the likelihood function to provide an exact approximation of an
idealized algorithm which computes the likelihood function exactly.
The algorithm simply consists of replacing the true value of the likelihood
function required to implement the standard Metropolis--Hastings (MH)
algorithm with estimators, but is nevertheless guaranteed to be correct
in that it leaves the required distribution of interest marginally
invariant. In PMMH, the estimator of the likelihood is a byproduct
of a sequential Monte Carlo (SMC) algorithm, whose accuracy can be
improved by increasing $N$. 

In contrast, the particle Gibbs (PGibbs) sampler~\citep{andrieu-doucet-holenstein}
involves approximating a Gibbs sampler which consists of constructing
a Markov chain $\{(\theta_{i},X_{i}),i\geq0\}$, by repeatedly sampling
from $\pi_{\theta}\big({\rm d}x\big)$ and $\pi_{x}\big({\rm d}\theta\big)$
in turn. In practice sampling from $\pi_{\theta}\big({\rm d}x\big)$
may be particularly difficult and the conditional SMC (cSMC)~\citep{andrieu-doucet-holenstein}
update is a Markov transition probability $P_{N,\theta}$ which leaves
$\pi_{\theta}\big({\rm d}x\big)$ invariant, therefore allowing the
implementation of a Metropolis-within-Gibbs algorithm, that is a Markov
transition probability leaving $\pi\big({\rm d}\theta\times{\rm d}x\big)$
invariant. The cSMC relies for its construction, as suggested by its
name, on an SMC-like procedure and it is expected that as $N$ increases
$P_{N,\theta}$ approaches $\pi_{\theta}\big({\rm d}x\big)$.

While PMMH methods have been studied in a series of papers~\citep{andrieu-roberts,andrieu2015,lee-latuszynski,sherlock2015,Doucet07032015},
a theoretical study of the PGibbs is still missing. Indeed it has
been shown that as $N$ increases, performance of the PMMH approaches
that of the exact MH algorithm but the question of the approximation
of the Gibbs sampler by a PGibbs has not been addressed to date. We
note however that a study of one of its components, the cSMC update,
has recently been undertaken in~\citep{chopin:singh:2013}, in which
a coupling argument is central to their analysis. We refer to the
Markov chain obtained by iterating the cSMC algorithm for a fixed
target distribution as iterated i-cSMC here in order to distinguish
it from that of the PGibbs. The present manuscript addresses questions
concerning the i-cSMC similar to those of~\citep{chopin:singh:2013},
but our results differ in many respects and complement their findings
in several directions. At a technical level our approach seems to
be more straightforward in the scenario considered, relies on weaker
assumptions for uniform convergence which we prove are necessary and
sufficient and lead to quantitative bounds on performance measures
in terms of the number $N$ of particles involved. We additionally
transfer sufficient conditions for uniform ergodicity of the i-cSMC
Markov chain into sufficient conditions for geometric ergodicity of
the associated PGibbs Markov chain, the main motivation behind our
work. This allows us in particular to show that under some conditions
PGibbs is asymptotically as efficient as the Gibbs sampler as the
number $N$ of particles increases.

Contemporary to the first version of the present manuscript~\citep{andrieu2013uniform},~\citep{lindsten_pg}
have also provided essentially the same sufficient conditions for
the uniform convergence of the i-cSMC Markov chain (Theorem~\ref{thm:THEtheorem},
Section~\ref{sec:The-i-CSMC}) using a different proof technique.
Here we have further established that the aforementioned conditions
are also necessary for uniform convergence in general, but also geometric
ergodicity in many realistic scenarios (Section~\ref{sec:conjectureonboundedness}).
Similarly to us~\citep{lindsten_pg} also provide quantitative bounds
and associated scaling properties of the i-cSMC, albeit for a different
set of specialised conditions (a detailed comparison of the assumptions
is provided after Theorem~\ref{thm:quantitative-on-epsilon_N} at
the end of Section~\ref{sec:The-i-CSMC}). We have also very recently
become aware of the contribution~\citep{del2014feynman} to the analysis
of the properties of the cSMC, established using the formalism of~\citep{delmoral:2004},
but their practical implications are unclear. Similarly to~\citep{chopin:singh:2013},~\citep{lindsten_pg}
do not attempt to address the practically important question of how
uniform ergodicity of the i-cSMC can be translated into geometric
ergodicity of the PGibbs sampler, an issue we address in Section~\ref{sec:The-particle-Gibbs}.
In Section~\ref{sec:Discussion} we contrast the results obtained
in this paper concerning the i-cSMC and PGibbs algorithm with known
results concerned with other particle MCMC methods and draw final
conclusions.

Similarly to SMC methods, the cSMC and associated algorithms are complex
mathematical objects which require the introduction of sometimes overwhelming
notation which may obscure the main ideas. In the next section we
attempt to remedy this by presenting our results in a simplified scenario,
which captures our main ideas, before moving on to the general scenario.

\section{Statement of our results in a simplified scenario}

We first explain our results on a particularly simple instance of
the i-cSMC algorithm. This should provide the reader with the essence
of the results proved later on in the general scenario, while its
simple structure will allow us to outline the main idea behind our
proof in the general set-up (in Section~\ref{sec:Minorization-and-Dirichlet}).

Assume we are interested in sampling from a probability distribution
$\pi$ on some measurable space $\bigl(\mathsf{X},\mathcal{B}\bigl(\mathsf{X}\bigr)\bigr)$.
We define the probability distribution $\tilde{\pi}$ on $\{1,\ldots,N\}\times\mathsf{X}^{N}$
\begin{equation}
\tilde{\pi}\left(k,{\rm d}z^{1:N}\right)=\frac{1}{N}\pi({\rm d}z^{k})\prod_{j=1,j\neq k}^{N}M({\rm d}z^{j})\quad,\label{eq:artificialforiSIR}
\end{equation}
for some probability distribution $M$ defined on $\bigl(\mathsf{X},\mathcal{B}\bigl(\mathsf{X}\bigr)\bigr)$
and such that for any $S\in\mathcal{B}\bigl(\mathsf{X}\bigr)$ such
that $\pi(S)>0$ then $M(S)>0$. As pointed out in the authors' discussion
reply of~\citep{andrieu-doucet-holenstein}, in this simple scenario
one can define an MCMC algorithm targeting $\pi$ by iterating the
classical sampling importance resampling (SIR) procedure. More specifically,
we sample alternately from (a) $Z^{1:N\setminus k}\mid\left(K=k,Z^{k}=z^{k}\right)\sim\prod_{i=1,i\neq k}^{N}M(z^{i})$
and (b) $K\mid\left(Z^{1:N}=z^{1:N}\right)\sim\tilde{\pi}(k|z^{1:N})$,
where $Z^{1:N\setminus k}:=\bigl(Z^{1},Z^{2},\ldots,Z^{k-1},Z^{k+1},\ldots,Z^{N}\bigr)$.
Owing to the fact that this algorithm is a Gibbs sampler on the distribution
above and from the standard interlacing property of the two stage
Gibbs sampler, one can check that the sequence $\{Z_{i}^{K_{i}}\}$
defines a Markov chain with invariant distribution $\pi$, and that
its transition kernel is for any $(x,S)\in\mathsf{X}\times\mathcal{B}\bigl(\mathsf{X}\bigr)$
\[
P_{N}(x,S)=\int_{\mathsf{X}^{N-1}}\sum_{k=1}^{N}\frac{G(z^{k})}{\sum_{j=1}^{N}G(z^{j})}\mathbb{I}\{z^{k}\in S\}\prod_{i=2}^{N}M(\mathrm{d}z^{i})
\]
with $G(x):=\pi({\rm d}x)/M({\rm d}x)$ and the convention $z^{1}=x$.
Our first results are concerned with properties of the homogeneous
Markov chain with transition probability $P_{N}$, in terms of $\bar{G}:=\pi-{\rm ess}\sup_{x}G(x)$
and $N$. We refer to the resulting algorithm as iterated SIR (i-SIR).

We briefly introduce notions that allow us to make quantitative statements
about the Markov chains under study. We use classical Hilbert space
techniques for the analysis of reversible Markov chains. Letting $\mu\bigl(\cdot\bigr)$
be a probability distribution defined on some measurable space $\bigl(\mathsf{E},\mathcal{B}\bigl(\mathsf{E}\bigr)\bigr)$,
we define the function space
\[
L^{2}(\mathsf{E},\mu):=\left\{ f:\mathsf{E}\rightarrow\mathbb{R}:\mu(f^{2})<\infty\right\} ,
\]
where the functions are taken to be measurable; hereafter all functions
considered are assumed to be measurable with respect to an appropriate
$\sigma$-algebra. Let $\Pi:\mathsf{E}\times\mathcal{B}\bigl(\mathsf{E}\bigr)\rightarrow[0,1]$
be a $\mu$-reversible Markov transition kernel and let $\{\xi_{i},i\geq0\}$
be the stationary Markov chain with transition kernel $\Pi$ (such
that $\xi_{0}\sim\mu$). We will use the standard notation for any
probability distribution $\nu$ on $\bigl(\mathsf{E},\mathcal{B}\bigl(\mathsf{E}\bigr)\bigr)$
and measurable function $f:\mathsf{E}\rightarrow\mathbb{R}$, 
\[
\nu\bigl(f\bigr):=\int_{\mathsf{E}}f(x)\nu({\rm d}x)\quad\text{,}\quad\Pi f(x):=\int_{\mathsf{E}}f(y)\Pi\bigl(x,{\rm d}y\bigr)\quad,
\]
for $k\geq2$, by induction,
\[
\Pi^{k}f(x):=\int_{\mathsf{E}}\Pi\bigl(x,{\rm d}y\bigr)\Pi^{k-1}f(y)\quad.
\]
We denote $\nu\Pi^{k}f:=\nu\bigl(\Pi^{k}f\bigr)$ and refer to $\nu\Pi^{k}$
as either a probability measure or its corresponding operator on $L^{2}(\mathsf{E},\mu)$.
For $f\in L^{2}\bigl(\mathsf{E},\mu\bigr)$, we define the variance
of $f$ under $\mu$ as ${\rm var}_{\mu}(f):=\mu(f^{2})-\mu(f)^{2}$
and the ``asymptotic variance'' of $M^{-1}\sum_{i=1}^{M}f\big(\xi_{i}\big)$
for stationary realizations $\{\xi_{i},i\geq0\}$ associated to the
homogeneous Markov chain with transition $\Pi$ as 
\[
\mathrm{var}(f,\Pi):=\lim_{M\rightarrow\infty}\mathrm{var}\left({\textstyle M^{-1/2}}{\textstyle \sum}_{i=1}^{M}[f(\xi_{i})-\mu(f)]\right)\quad.
\]
Some of our results involve norms of signed measures. As in, e.g.,
\citep{roberts-rosenthal-geometric}, for any signed measure $\nu$
on $\bigl(\mathsf{E},\mathcal{B}\bigl(\mathsf{E}\bigr)\bigr)$ we
let 
\[
\|\nu\|_{TV}:=\frac{1}{2}\sup_{f:\mathsf{E}\rightarrow[-1,1]}\nu\bigl(f\bigr)
\]
denote the total variation distance and for $\nu\ll\mu$, 
\begin{equation}
\Vert\nu\Vert_{L^{2}(\mathsf{E},\mu)}^{2}:=\int_{\mathsf{E}}\left|\frac{{\rm d}\nu}{{\rm d}\mu}\right|^{2}{\rm d}\mu=\sup_{f\in L^{2}\bigl(\mathsf{E},\mu\bigr),\:\|f\|_{\mu}>0}\frac{|\nu(f)|}{\|f\|_{\mu}}\quad.\label{eq:defL2norm}
\end{equation}
denote the $L^{2}(\mathsf{E},\mu)$ norm. 

Our results can be summarized as follows
\begin{enumerate}
\item $P_{N}$ is reversible with respect to $\pi$ and positive, that is
the i-SIR Markov chain has non-negative stationary autocorrelations. 
\item If $\bar{G}<\infty$, and $N\geq2$, the i-SIR Markov chain is uniformly
ergodic with for any $x\in\mathsf{X}$,
\[
\|P_{N}^{n}(x,\cdot)-\pi(\cdot)\|_{TV}\leq\left(1-\frac{N-1}{2\bar{G}+N-2}\right)^{n}\quad.
\]

\item If $\bar{G}<\infty$, then for any $f\in L^{2}(\mathsf{X},\pi)$,
\[
\mathrm{var}{}_{\pi}(f)\leq\mathrm{var}(f,P_{N})\leq\left[2\left(1+\frac{2\bar{G}-1}{N-1}\right)-1\right]\mathrm{var}{}_{\pi}(f)\quad.
\]

\item If $\bar{G}=\infty$ then the i-SIR Markov chain cannot be geometrically
ergodic for any finite $N$.
\end{enumerate}
The second and third points provide quantitative bounds on standard
measures of performance for MCMC algorithms, where the second provides
a bound on the uniform (or equivalently uniformly geometric) rate
of convergence of the Markov chain. Interest in algorithms such as
i-SIR is motivated empirically from observed behaviour in line with
the above bounds, as performance improves as $N$ increases, and part
of our purpose here is to confirm and quantify theoretically such
empirical successes. Moreover, this improvement can often be obtained
with little extra computational effort, since on a parallel architecture
one can sample from $M$ and evaluate $G$ in parallel, a characteristic
of SMC algorithms more generally~\citep{lee-yau-giles-doucet-holmes}.

While i-SIR can be used alone to sample from fairly general distributions,
it can also be used as a constituent element of more elaborate MCMC
schemes. Assume now that we wish to sample from a distribution $\pi$
defined on some measurable space $\left(\Theta\times\mathsf{X},\mathcal{B}(\Theta)\times\mathcal{B}\bigl(\mathsf{X}\bigr)\right)$,
often defined for some $S\in\mathcal{B}(\Theta)\times\mathcal{B}(\mathsf{X})$
via (note the different nature of $\pi$ as compared to earlier)
\[
\pi(S):=\frac{\int_{S}G_{\theta}(x)M_{\theta}({\rm d}x)\varpi({\rm d}\theta)}{\int_{\Theta\times\mathsf{X}}G_{\theta}(x)M_{\theta}({\rm d}x)\varpi({\rm d}\theta)}\quad,
\]
where $\{G_{\theta},\theta\in\Theta\}$ is a collection of non-negative
potential functions and $\{M_{\theta},\theta\in\Theta\}$ a collection
of probability measures which define for each $\theta\in\Theta$ the
conditional distributions $\pi_{\theta}\bigl({\rm d}x\bigr):=M_{\theta}\bigl({\rm d}x\bigr)G_{\theta}(x)/\gamma_{\theta}$
with
\[
\gamma_{\theta}:=\int_{\mathsf{X}}G_{\theta}(x)M_{\theta}({\rm d}x)\quad.
\]
The interpretation in a statistical context is that $\varpi$ is the
prior distribution for some parameter $\theta$ of interest, whilst
$\gamma_{\theta}$ is the likelihood function associated with some
observed data and $x$ corresponds to the so-called latent variable(s).
The form of $\gamma_{\theta}$ is often derived from the data being
explained by the latent variable $x$ whose \emph{a priori} distribution
conditional upon $\theta$ is $M_{\theta}$ and the likelihood function
given the data and $x$ is $G_{\theta}(x)$. Assume here that we are
able to sample from $\pi_{x}$, the conditional distribution of $\theta$
given $X=x$. For any $\theta\in\Theta$ one can define the i-SIR
kernel for any $(x,S)\in\mathsf{X}\times\mathcal{B}\bigl(\mathsf{X}\bigr)$
via
\[
P_{N,\theta}(x,S)=\int_{\mathsf{X}^{N-1}}\sum_{k=1}^{N}\frac{G_{\theta}(z^{k})}{\sum_{j=1}^{N}G_{\theta}(z^{j})}\mathbb{I}\{z^{k}\in S\}\prod_{i=2}^{N}M_{\theta}(\mathrm{d}z^{i})\quad,
\]
with $z^{1}=x$, so that the invariant distribution associated with
$P_{N,\theta}$ is $\pi_{\theta}$, the conditional distribution of
$X$ given $\theta$. One can sample from $\pi({\rm d}\theta\times{\rm d}x)$
with the following Markov transition, defined for any $\bigl(\theta_{0},x,S\bigr)\in\Theta\times\mathsf{X}\times\big(\mathcal{B}(\Theta)\times\mathcal{B}(\mathsf{X})\big)$
via 
\[
\Phi{}_{N}(\theta_{0},x;S):=\int_{S}P_{N,\theta}(x,{\rm d}y)\pi_{x}({\rm d}\theta)\quad,
\]
which can be viewed as an exact approximation of the Gibbs sampler
defined via
\[
\Gamma(\theta_{0},x;S):=\int_{S}\pi_{\theta}({\rm d}y)\pi_{x}({\rm d}\theta)\quad.
\]
The term exact approximation refers to the fact that while $P_{N,\theta}$
can be thought of as an approximation of the conditional distribution
$\pi_{\theta}$ the resulting algorithm converges to $\pi$ and can
be made arbitrarily close to $\Gamma$ as we increase $N$ as explained
below -- we will refer to this algorithm and its generalisation as
the particle Gibbs (PGibbs) sampler. Throughout the paper we will
use the following convention: we will say $f\in L^{2}\bigl(\mathsf{E},\pi\bigr)$
with $\mathsf{E}=\Theta$ (resp. $\mathsf{E}=\mathsf{X}$) to mean
that $f:\mathsf{E}\rightarrow\mathbb{R}$ is square integrable under
the relevant marginal of $\pi$, or $f:\Theta\times\mathsf{X}\rightarrow\mathbb{R}$
does not depend on $x$ (resp. $\theta$) and is square integrable
under the relevant marginal of $\pi$. This should not lead to any
possible confusion. Letting $\bar{G}:=\pi-{\rm ess}\sup_{\theta,x}\frac{G_{\theta}(x)}{\gamma_{\theta}}$,
our results for the PGibbs sampler, are as follows
\begin{enumerate}
\item Assume the $\Gamma$ Markov chain is such that there exists $\beta\in(0,1]$
such that for any $f:\mathsf{X}\rightarrow[-1,1]$ and $\nu\ll\pi$
\[
\left|\nu\Gamma^{n}(f)-\pi(f)\right|\leq\Vert\nu-\pi\Vert_{L^{2}(\mathsf{X},\pi)}\left(1-\beta\right)^{n}\quad.
\]
If $\bar{G}<\infty$, and $N\geq2$, then for any $f:\mathsf{X}\rightarrow[-1,1]$
and $\nu\ll\pi$
\[
\left|\nu\Phi_{N}^{n}(f)-\pi(f)\right|\leq\Vert\nu-\pi\Vert_{L^{2}(\mathsf{X},\pi)}\left(1-\beta_{N}'\right)^{n}\quad,
\]
where $\beta_{N}'$ satisfies
\[
\beta_{N}'\geq\frac{N-1}{2\bar{G}+N-2}\beta\quad.
\]

\item For any $f\in L^{2}(\mathsf{X},\pi)$ and $N\geq2$, the asymptotic
variance ${\rm var}(f,\Phi_{N})$ satisfies 
\[
{\rm var}\bigl(f,\Gamma\bigr)\leq{\rm var}(f,\Phi_{N})\leq\frac{2\bar{G}-1}{N-1}{\rm var}_{\pi}(f)+\left(1+\frac{2\bar{G}-1}{N-1}\right){\rm var}(f,\Gamma)\,.
\]

\item For any $f\in L^{2}\bigl(\Theta,\pi\bigr)$ and $N\geq2$, the asymptotic
variance ${\rm var}(f,\Phi_{N})$ satisfies 
\[
{\rm var}\bigl(f,\Gamma\bigr)\leq{\rm var}(f,\Phi_{N})\leq\left(1+\frac{2\bar{G}-1}{N-1}\right){\rm var}\bigl(f,\Gamma\bigr)-\left(\frac{2\bar{G}-1}{N-1}\right){\rm var}_{\pi}(f)\quad.
\]

\end{enumerate}
In the sequel, we prove similar results in the more general (and complex)
scenario where $P_{N,\theta}$ is defined by a general cSMC algorithm
with multinomial resampling, but the key ideas and results are similar
(Section~\ref{sec:The-i-CSMC}). The results concerning the general
form of the PGibbs sampler, from which its convergence in the sense
of points 1--3 above follows, can be found in Section~\ref{sec:The-particle-Gibbs}.

\section{The i-cSMC and its properties\label{sec:The-i-CSMC}}

We mostly follow the notation of~\citep{delmoral:2004} and use the
following conventions for lists, indices and superscripts. For $N\in\mathbb{N}$,
we denote $[N]:=\bigl\{1,\ldots,N\bigr\}$, and for any $p\in\mathbb{N}$,
$\mathbf{k},\mathbf{l}\in[N]^{p}$ and $u_{k}^{l}:\mathbb{N}^{2}\rightarrow\mathsf{E}$
(for a generic set $\mathsf{E}$ dependent on the context) we will
use the notation $u_{\mathbf{k}}^{\mathbf{l}}$ to mean $\bigl(u_{k_{1}}^{l_{1}},u_{k_{2}}^{l_{2}},\dots,u_{k_{p}}^{l_{p}}\bigr)$,
and whenever there is no dependence on $l$ (resp. $k$) of $u_{k}^{l}$
we simply ignore this superscript (resp. this index). We will also
use the notation, for $k,l\in\mathbb{N}$ such that $l\geq k$, $k:l:=\bigl(k,k+1,\ldots,l\bigr)$.
Let $\bigl(\mathsf{Z},\mathcal{B}\bigl(\mathsf{Z}\bigr)\bigr)$ be
a measurable space and for some $T\geq1$ define a family of Markov
transition probabilities on this space $\bigl\{ M_{t}\bigl(\cdot,\cdot\bigr),t\in[T]\bigr\}$
with the convention that for $t=1$ and any $z\in\mathsf{Z}$, $M_{1}(z,{\rm d}u)=M_{1}({\rm d}u)$
and a family of measurable non-negative functions, the potentials
$G_{t}:\mathsf{Z}\rightarrow[0,\infty)$, again for $t\in[T]$. We
first define an inhomogeneous Markov chain $\{Z_{1},\ldots,Z_{T}\}$
on $\mathsf{X}:=\mathsf{Z}^{T}$ endowed with the product $\sigma-$algebra
$\mathcal{B}\bigl(\mathsf{X}\bigr)=\mathcal{B}(\mathsf{Z})^{T}$ and
with probability distribution $\mathbb{P}\bigl(\cdot\bigr)$ and associated
expectation $\mathbb{E}\bigl(\cdot\bigr)$ such that for $t=1$, the
initial distribution is $\mathbb{P}\left(Z_{1}\in{\rm d}z_{1}\right):=M_{1}({\rm d}z_{1})$,
and for $t=2,\ldots,T$ the transition probability is given by $M_{t}$,
i.e. 
\[
\qquad\mathbb{P}\left(Z_{t}\in{\rm d}z_{t}\middle|Z_{t-1}=z_{t-1}\right):=M_{t}(z_{t-1},{\rm d}z_{t})\quad.
\]
We define for $p\in[T]$ and $f_{p}:\mathsf{Z}^{p}\rightarrow\mathbb{R}$
\[
\gamma_{p}(f_{p}):=\mathbb{E}\left(f_{p}(Z_{1},\ldots,Z_{p})\prod_{t=1}^{p}G_{t}\bigl(Z_{t}\bigr)\right),
\]
and can define for any $S\in\mathcal{B}(\mathsf{X})$ the probability
distribution $\pi$ (which will be the target distribution of interest)
\begin{equation}
\pi(S):=\frac{\gamma_{T}(\mathbb{I}\bigl\{\,\cdot\,\in S\bigr\})}{\gamma_{T}}\quad,\label{eq:defofpiforSMCframework}
\end{equation}
where $\mathbb{I}\bigl\{\cdot\bigr\}$ denotes the indicator function
and $\gamma_{T}:=\gamma_{T}\bigl(1\bigr)$. For $l>k\ge0$, we define
\[
M_{k,l}(z_{k},{\rm d}z_{k+1:l}):=\prod_{t=k+1}^{l}M_{t}(z_{t-1},{\rm d}z_{t})\quad.
\]
Note in particular that with the convention above, for any $l\geq2$
and $z_{0}\in\mathsf{Z}$, $M_{0,l}\bigl(z_{0},{\rm d}z_{1:l}\bigr):=M_{1}({\rm d}z_{1})\times M_{1,l}\bigl(z_{1},{\rm d}z_{2:l}\bigr)$.

The iterated conditional SMC (i-cSMC) is a family of homogeneous Markov
chains, with state-space $\bigl(\mathsf{X},\mathcal{B}\bigl(\mathsf{X}\bigr)\bigr)$,
indexed by $N\in\mathbb{N}$ (the concrete meaning of $N$ shall become
clearer below). We denote by $P_{N}\bigl(\cdot,\cdot\bigr):\mathsf{X}\times\mathcal{B}\bigl(\mathsf{X}\bigr)\rightarrow[0,1]$
the corresponding Markov transition kernels, which we now define.
To that end, we first detail for any $N\in\mathbb{N}$ the probability
distribution of the conditional SMC (cSMC) algorithm, which corresponds
to a process defined on the extended space $\mathsf{W}:=\bigl(\mathsf{Z}^{N}\times[N]^{N}\bigr)^{T-1}\times\mathsf{Z}^{N}\times[N]$
endowed with the corresponding product $\sigma-$algebra $\mathcal{B}\bigl(\mathsf{W}\bigr)$,
of which $P_{N}$ is a simple by-product. Our focus is on a particular
implementation of the algorithm corresponding to ``multinomial resampling''--other
schemes are considered in~\citep{chopin:singh:2013}. For any $x\in\mathsf{X}$
and with $\mathbf{1}\in\{1\}^{T}$ we define the process $\{Z_{t},A_{t},t=1,\ldots,T\}$
on $\mathsf{W}$ through 
\begin{align}
\mathbb{P}_{\mathbf{1},x}^{N}\left(Z_{1}\in{\rm d}z_{1}\right): & =\delta_{x_{1}}({\rm d}z_{1}^{1})\prod_{i=2}^{N}M_{1}({\rm d}z_{1}^{i})\label{eq:def_P_=00007B1,x=00007D-time-1}
\end{align}
 and for $t\in\{2,\ldots,T\}$
\begin{eqnarray}
 &  & \mathbb{P}_{\mathbf{1},x}^{N}\big(Z_{t}\in{\rm d}z_{t},A_{t-1}=a_{t-1}\,\big|\,Z_{1:t-1}=z_{1:t-1},A_{1:t-2}=a_{1:t-2}\big)\nonumber \\
 &  & \mathbb{=P}_{\mathbf{1},x}^{N}\left(Z_{t}\in{\rm d}z_{t},A_{t-1}=a_{t-1}\left|Z_{t-1}=z_{t-1}\right.\right)\nonumber \\
 &  & =\delta_{x_{t}}({\rm d}z_{t}^{1})\mathbb{I}\{a_{t-1}^{1}=1\}\prod_{i=2}^{N}\Bigg(\sum_{k=1}^{N}\frac{G_{t-1}(z_{t-1}^{k})}{\sum_{j=1}^{N}G_{t-1}(z_{t-1}^{j})}\mathbb{I}\left\{ a_{t-1}^{i}=k\right\} M_{t}(z_{t-1}^{k},{\rm d}z_{t}^{i})\Bigg)\quad,\label{eq:def_P_=00007B1,x=00007D-other-times}
\end{eqnarray}
where we keep $k$ to emphasize that we are sampling from that mixture.
For the last iteration we only require one index and point out that
whereas $A_{t}\in[N]^{N}$ for $t=1,\ldots,T-1$, we have $A_{T}\in[N]$
following 
\[
\mathbb{P}_{\mathbf{1},x}^{N}\left(A_{T}=k\left|Z_{T}=z_{T}\right.\right)={\displaystyle {\textstyle \frac{{\displaystyle G_{T}(z_{T}^{k})}}{{\displaystyle {\textstyle \sum_{j=1}^{N}}G_{T}(z_{T}^{j})}}\quad.}}
\]

The stochastic process defined by $\mathbb{P}_{{\bf 1},x}^{N}$ is
referred to as the conditional SMC algorithm because it is closely
related to a standard SMC algorithm, but where $x$ is a ``fixed
path'' with lineage ${\bf 1}$. However, as remarked in~\citep{andrieu-doucet-holenstein},
$\mathbb{P}_{\mathbf{1},x}^{N}$ is not a conditional distribution
of $\mathbb{P}^{N}\bigl(\cdot\bigr)$, the standard SMC algorithm
whose definition here is deferred to~\citep[Appendix~\ref{sec:Comparison-with-Particle}]{andrieu2015uniformsupplement}.
We note further that in order to simplify presentation we have focused
here on the scenario where the lineage of $x$ was ${\bf 1}$ but
that we could also use, as in~\citep{andrieu2015uniformsupplement},
the cSMC with ${\bf k}\in[N]^{T}$ (with associated symbol $\mathbb{P}_{\mathbf{k},x}^{N}$
and $\mathbb{E}_{\mathbf{k},x}^{N}$) corresponding to the process
above, but where $\delta_{x_{t}}({\rm d}z_{t}^{1})\mathbb{I}\{a_{t-1}^{1}=1\}$
in (\ref{eq:def_P_=00007B1,x=00007D-other-times}) is replaced with
$\delta_{x_{t}}({\rm d}z_{t}^{k_{t}})\mathbb{I}\{a_{t-1}^{k_{t}}=k_{t-1}\}$
and $\delta_{x_{1}}({\rm d}z_{1}^{1})$ with $\delta_{x_{1}}({\rm d}z_{1}^{k_{1}})$
in (\ref{eq:def_P_=00007B1,x=00007D-time-1}).

For any $\mathbf{i}:=\bigl(i_{1},i_{2},\ldots,i_{T}\bigr)\in[N]^{T}$,
$z_{1:T}\in\bigl(\mathsf{Z}^{N}\bigr)^{T}$, $a_{1:T}:=(a_{1},\ldots,a_{T})\in\bigl([N]^{N}\bigr)^{T-1}\times[N]$
and $S\in\mathcal{B}\bigl(\mathsf{X}\bigr)$ define
\begin{equation}
I_{\mathbf{i}}\bigl(z_{1:T},a_{1:T},S\bigr):=\mathbb{I}\{z_{1:T}^{\mathbf{i}}\in S,i_{T}=a_{T}\}\prod_{t=1}^{T-1}\mathbb{I}\{i_{t}=a_{t}^{i_{t+1}}\}\quad.\label{eq:def_of_I_i}
\end{equation}
Then the transition kernel of the iterated conditional SMC (i-cSMC),
in the multinomial sampling scenario, is given for any $x\in\mathsf{X}$
and $S\in\mathcal{B}\left(\mathsf{X}\right)$ by 
\begin{equation}
P_{N}(x,S):=\mathbb{E}_{\mathbf{1},x}^{N}\left[{\textstyle \sum_{\mathbf{i}\in[N]^{T}}}I_{\mathbf{i}}\bigl(Z_{1:T},A_{1:T},S\bigr)\right]\quad,\label{eq:PN_defn}
\end{equation}
that is, conditional upon $x$ we consider the probability distribution
of those trajectories $Z_{1:T}^{\mathbf{i}}$ generated by the cSMC
which form a lineage compatible with the lineages defined by the random
variables $A_{1:T}$. Our main results concerning the i-cSMC algorithm
are the following (our results concerning the particle Gibbs sampler
are provided in Section~\ref{sec:The-particle-Gibbs}). We will denote
by $\pi_{t}$ the corresponding marginal distribution of $\pi$ (see
(\ref{eq:definitionpi_t:u}) for a precise definition).
\begin{thm}
\label{thm:THEtheorem}For $N\geq2$ the i-cSMC algorithm with kernel
$P_{N}$
\begin{enumerate}[label=(\alph*)]
\item is reversible with respect to $\pi$ and defines a positive operator,
\item \label{enu:GboundedResults}if for all $t\in\{1,\ldots,T\}$ $\pi_{t}-{\rm ess}\sup_{z_{t}}G_{t}(z_{t})<\infty$
then there exists $\epsilon_{N}>0$ such that

\begin{enumerate}[label=(\roman*)]
\item \label{enu:minorization} for any $(x,S)\in\mathsf{X}\times\mathcal{B}\bigl(\mathsf{X}\bigr)$,
\[
P_{N}(x,S)\geq\epsilon_{N}\pi(S)\quad,
\]
where $1-\epsilon_{N}=O(1/N)$,
\item \label{enu:uniformconvergence}for any probability distribution $\nu\ll\pi$
on $\bigl(\mathsf{X},\mathcal{B}(\mathsf{X})\bigr)$ and $k\geq1$
\[
\|\nu P_{N}^{k}\bigl(\cdot\bigr)-\pi\bigl(\cdot\bigr)\|_{L^{2}(\mathsf{X},\pi)}\leq\|\nu-\pi\|_{L^{2}(\mathsf{X},\pi)}(1-\epsilon_{N})^{k}\quad,
\]

\item for any $x\in\mathsf{X}$
\[
\|\delta_{x}P_{N}^{k}\bigl(\cdot\bigr)-\pi\bigl(\cdot\bigr)\|_{TV}\leq(1-\epsilon_{N})^{k}\quad,
\]

\item \label{enu:upperboundasymptvar}for any $f\in L^{2}\bigl(\mathsf{X},\pi\bigr)$
\[
\mathrm{var}{}_{\pi}(f)\leq\mathrm{var}(f,P_{N})\leq\left[2\epsilon_{N}^{-1}-1\right]\mathrm{var}{}_{\pi}(f)\quad.
\]

\end{enumerate}
\item \label{enu:iff-condition}if $\pi_{t}$-${\rm ess}\sup_{z_{t}}G_{t}(z_{t})=\infty$
for some $t\in[T]$, then, the i-cSMC kernel $P_{N}$ is not uniformly
ergodic for any $N\in\mathbb{N}$,
\item \label{enu:non-geometric-statement}if $\pi_{t}$-${\rm ess}\sup_{z_{t}}G_{t}(z_{t})=\infty$
for some $t\in[T]$ then, the i-cSMC kernel $P_{N}$ cannot be geometrically
ergodic for any $N\in\mathbb{N}$ if $\pi$ is equivalent to a Lebesgue
or counting measure on $\mathsf{X}$.
\end{enumerate}
\end{thm}
\begin{rem}
From Lemma~\ref{lem:full-supp}, statement~\ref{enu:non-geometric-statement}
holds under a more abstract assumption, but we have chosen this explicit
simplified statement for clarity at this point. In fact we suspect
that~\ref{enu:non-geometric-statement} holds under the assumption
$\pi_{t}-{\rm ess}\sup_{z_{t}}G_{t}(z_{t})=\infty$ for some $t\in[T]$
only, that is essential boundedness is a necessary condition for geometric
ergodicity; see Conjecture~\ref{conj:general-necessity}.
\end{rem}
With additional conditions on $\{M_{t},G_{t},t=1,\ldots\}$ one can
characterize $\epsilon_{N}$ in Theorem~\ref{thm:THEtheorem}\ref{enu:GboundedResults}
further, and in particular characterize the rate at which $N$ should
grow in terms of $T$ in order to maintain a set level of performance.
This also requires additional notation and following~\citep{delmoral:2004}
we define for any $z\in\mathsf{Z}$, $p,q\in\mathbb{N}$, $p\leq q$
and $f_{q}:\mathsf{Z}\rightarrow\mathbb{R}$,
\[
Q_{p,q}\bigl(f_{q}\bigr)\bigl(z\bigr):=\mathbb{E}\left[f_{q}\bigl(Z_{q}\bigr)\prod_{k=p}^{q-1}G_{k}\bigl(Z_{k}\bigr)\,\Bigg|\,Z_{p}=z\right]\quad,
\]
and with the convention $Q_{0,p}(f_{p})(x)=M_{1}Q_{1,p}(f_{p})$ for
any $f_{p}:\mathsf{Z}\rightarrow\mathbb{R}$, and

\[
\eta_{p}(f_{p}):=\frac{Q_{0,p}(f_{p})}{Q_{0,p}(1)}
\]
and $\bar{M}_{p,p+1}\bigl(z,\cdot\bigr)=M_{p+1}\bigl(z,\cdot\bigr)$
and for $q>p\geq0$ we have the recursive definition, for any $z_{p}\in\mathsf{Z}$,
\[
\bar{M}_{p,q}\bigl(z_{p},\cdot\bigr)=\int M_{p+1}(z_{p},{\rm d}z_{p+1})\bar{M}_{p+1,q}\bigl(z_{p+1},\cdot\bigr)\quad.
\]
The first condition is rather abstract, and can be viewed as a condition
on the $h$-functions investigated in~\citep{whiteley_stability}
in the context of stability properties of standard SMC algorithms.
\begin{condition}
\label{hyp:mixingabstract}There exists a constant $\alpha>0$ such
that for any $p,k\in\mathbb{N}$, 
\[
\sup_{z\in\mathsf{Z}}\frac{Q_{p,p+k}(1)(z)}{\eta_{p}Q_{p,p+k}(1)}\leq\alpha\quad.
\]

\end{condition}

\noindent One can however show that (A\ref{hyp:mixingabstract})
is implied by the following stronger assumption (see Lemma~\ref{lem:A2impliesA1}).
\begin{condition}[Strong mixing conditions]
\label{hyp:strongmixingpotentialassumptions}There exists $m\in\mathbb{Z}_{+}$
such that 
\begin{enumerate}[label=(\alph*)]
\item \label{hyp:enu:Mcondition}There exists a constant $1\leq\beta<\infty$
such that for any $p\geq1$ and any $(z,z')\in\mathsf{Z}$ and $S\in\mathcal{B}(\mathsf{Z})$,
\[
\bar{M}_{p,p+m}(z,S)\leq\beta\bar{M}_{p,p+m}(z',S)\;.
\]

\item \label{hyp:enu:Gcondition}The potential functions $G_{p}$ satisfy,
for some $\delta<\infty$,
\[
1\leq\sup_{z,z'\in\mathsf{Z}^{2},p\in\{1,\ldots,T\}}\frac{G_{p}(z)}{G_{p}(z')}\leq\delta^{1/m}\quad.
\]

\end{enumerate}
\end{condition}
\begin{thm}
\label{thm:quantitative-on-epsilon_N}Assume that for all $t\in\mathbb{N}$
$\pi_{t}-{\rm ess}\sup_{z_{t}}G_{t}(z_{t})<\infty$ and (A\ref{hyp:mixingabstract})
(or the stronger assumption (A\ref{hyp:strongmixingpotentialassumptions}))
holds. Then with $\epsilon_{N}$ as in Theorem~\ref{thm:THEtheorem}\ref{enu:GboundedResults}
for any $N\geq2$, there exists $C,\varepsilon>0$ such that with
$N=C\times T$, then for any $T\geq1$, $\epsilon_{N}\geq\varepsilon>0$.\end{thm}
\begin{rem}
Similar results for the PGibbs sampler are provided in Section~\ref{sec:The-particle-Gibbs}.
\end{rem}

\begin{proof}[Proof of Theorem~\ref{thm:THEtheorem} ]
The proofs of the various results are the subject of the following
sections. More specifically, statement 
\begin{enumerate}[label=(\alph*)]
\item follows from Lemma~\ref{lem:selfajoint_positive} (the latter property
was established in~\citep{chopin:singh:2013} and the former noted/proved
in~\citep{andrieu-doucet-holenstein,chopin:singh:2013}),
\item all parts follow from Corollary~\ref{cor:convergence_epsilon} and~\citep[Proposition~\ref{prop:boundsspectralgapandvarianceforP_N}]{andrieu2015uniformsupplement},
which gathers generic results on $\pi-$invariant Markov chains satisfying
\ref{enu:GboundedResults}\ref{enu:minorization},
\item follows from Proposition~\ref{prop:non-uniform},
\item follows from Proposition~\ref{prop:non-geom-easy} and Lemma~\ref{lem:full-supp};
Remark~\ref{rem:strictly_positive}.
\end{enumerate}
\end{proof}

\begin{proof}[Proof of Theorem~\ref{thm:quantitative-on-epsilon_N} ]
Follows from Proposition~\ref{prop:mixing_bound}, Corollary~\ref{cor:linear_in_T_epsilon_N_bound}
and Lemma~\ref{lem:A2impliesA1}.
\end{proof}
As pointed out in the introduction, soon after completing this work
we have become aware of~\citep{lindsten_pg}, where a subset of our
results have also been independently discovered. This motivates the
following comparison. Result~\ref{enu:GboundedResults}\ref{enu:minorization}
of Theorem~\ref{thm:THEtheorem} is identical to Theorem~1 of~\citep{lindsten_pg},
but relies on a different proof. Results~\ref{enu:GboundedResults}\ref{enu:uniformconvergence}--\ref{enu:upperboundasymptvar}
rely on standard arguments, although~\ref{enu:upperboundasymptvar}
does not seem to be well known and establishes informative quantitative
bounds. The study of the necessity of our conditions to imply uniform
or geometric ergodicity is not addressed in~\citep{lindsten_pg}.
The result of Theorem~\ref{thm:quantitative-on-epsilon_N} corresponds
to Proposition~5 of~\citep{lindsten_pg}. The conditions under which
Theorem~\ref{thm:quantitative-on-epsilon_N} holds are rather stringent
for some applications, in particular in the state-space model scenario.
As discussed by~\citep{lindsten_pg} in that scenario (A\ref{hyp:strongmixingpotentialassumptions})
will essentially only hold in the case where $\mathsf{X}$ is compact.
The condition (A\ref{hyp:mixingabstract}) is weaker and more natural
in our analysis, but is not currently easy to verify in applications
except through (A\ref{hyp:strongmixingpotentialassumptions}).

In an attempt to relax (A\ref{hyp:strongmixingpotentialassumptions}),
the authors of~\citep{lindsten_pg} investigate another set of specialised
assumptions guaranteeing that the result of Theorem~\ref{thm:quantitative-on-epsilon_N}
holds even in some non-compact scenarios provided the number of particles
$N$ grows at a rate $T^{1/\gamma}$ for any $\gamma\in(0,1)$, a
result in line with what is obtained with the stronger assumption
(A\ref{hyp:strongmixingpotentialassumptions}), for which $\gamma=1$
is permissible. This requires the specification of a ``moment assumption''
which aims at controlling the variations of the various quantities
involved under the law of the observation process $\{Y_{t},t\geq0\}$.
Their approach, however, does not seem to allow one to consider the
scaling properties of the PGibbs sampler (i.e. not just the i-cSMC);
see their Theorem~6 and Remark~7.  More importantly we note that
their results require the law of the data to coincide with that of
the specified model for some $\theta^{\star}\in\Theta$ which, although
suggestive of what may happen in practice, is always an idealization.
This delicate work is the main focus of the remainder of their investigation
while here, in addition to establishing the necessity of some of the
conditions, we have focused on the transference of the results obtained
for the i-cSMC to the PGibbs sampler (Section~\ref{sec:The-particle-Gibbs})
with the aim of showing that the PGibbs has performance inferior to
that of the Gibbs sampler, but arbitrarily close if we increase $N$.

\section{Establishing the uniform minorization condition\label{sec:Minorization-and-Dirichlet}}

Before proceeding we turn to the i-SIR which is particularly simple
to analyze. The reason for detailing the short analysis of this simple
scenario is to provide the reader with an overview of the developments
which are to follow -- the remainder of the paper essentially replicates
the key steps of the argument below, albeit in the more complex SMC
framework. Notice that in this scenario $\mathsf{X}=\mathsf{Z}$ since
$T=1$. We let $G(x):=\pi({\rm d}x)/M({\rm d}x)$ for any $x\in\mathsf{X}$
and assume that $\bar{G}:=\sup_{x\in\mathsf{X}}G(x)<\infty$. Then
for $(x,S)\in\mathsf{X}\times\mathcal{B}\bigl(\mathsf{X}\bigr)$ we
can rewrite \allowdisplaybreaks[4] 
\begin{eqnarray*}
P_{N}(x,S) & = & \sum_{k=1}^{N}\int_{\mathsf{X}^{N}}{\textstyle {\displaystyle \frac{\pi({\rm d}z^{k})/M({\rm d}z^{k})}{\sum_{j=1}^{N}G(z^{j})}}}\mathbb{I}\bigl\{ z^{k}\in S\bigr\}\left({\textstyle \delta_{x}({\rm d}z^{1})\prod_{i=2}^{N}}M(\mathrm{d}z^{i})\right)\\
 & = & \int_{\mathsf{X}^{N}}\frac{1}{\sum_{j=1}^{N}G(z^{j})}\frac{\pi({\rm d}z^{1})}{M({\rm d}z^{1})}\mathbb{I}\bigl\{ z^{1}\in S\bigr\}\left({\textstyle \delta_{x}({\rm d}z^{1})\prod_{i=2}^{N}}M(\mathrm{d}z^{i})\right)\\
 &  & +\sum_{k=2}^{N}\int_{\mathsf{X}^{N}}\frac{1}{\sum_{j=1}^{N}G(z^{j})}\mathbb{I}\bigl\{ z^{k}\in S\bigr\}\pi({\rm d}z^{k})\left({\textstyle \delta_{x}({\rm d}z^{1})\prod_{i=2,i\neq k}^{N}}M(\mathrm{d}z^{i})\right)\\
 & = & \sum_{k=1}^{N}\int_{\mathsf{X}}\mathbb{E}_{1,x,k,y}\left[\frac{\mathbb{I}\bigl\{ y\in S\bigr\}}{\sum_{j=1}^{N}G(Z^{j})}\right]\left(\mathbb{I}\{k=1\}\frac{\pi({\rm d}x)}{M({\rm d}x)}\delta_{x}({\rm d}y)+\mathbb{I}\{k\neq1\}\pi({\rm d}y)\right),
\end{eqnarray*}
where $\mathbb{E}_{1,x,k,y}\bigl(\cdot\bigr)$ defines an expectation
for the random variables $Z^{1},\ldots,Z^{N}$ associated to the probability
distribution
\[
{\textstyle \delta_{x}({\rm d}z^{1})\prod_{i=2}^{N}}M(\mathrm{d}z^{i})
\]
for $k=1$ and $x=y\in\mathsf{X}$, and
\[
{\textstyle \delta_{x,y}({\rm d}z^{1}\times{\rm d}z^{k})\prod_{i=2,i\neq k}^{N}}M(\mathrm{d}z^{i})
\]
for $k\neq1$ and $x,y\in\mathsf{X}$. This auxiliary process turns
out to be central to our analysis, and will be generalised to the
general scenario and called ``doubly'' cSMC (c$^{2}$SMC). Indeed,
omitting the term $k=1$ in the representation of $P_{N}$ and by
application of Jensen's inequality to the convex mapping $x\mapsto(x+a)^{-1}$
for $x,a\in\mathbb{R}_{+}$ we obtain
\begin{eqnarray*}
P_{N}(x,S) & \geq & \sum_{k=2}^{N}\int_{\mathsf{X}}\mathbb{I}\bigl\{ y\in S\bigr\}\mathbb{E}_{1,x,k,y}\left[\frac{1}{G(x)+G(y)+\sum_{j=2,j\neq k}^{N}G(Z^{j})}\right]\pi({\rm d}y)\\
 & \geq & \sum_{k=2}^{N}\int_{\mathsf{X}}\frac{\mathbb{I}\bigl\{ y\in S\bigr\}}{G(x)+G(y)+N-2}\pi({\rm d}y)\\
 & \geq & \frac{N-1}{2\bar{G}+N-2}\pi(S)\quad.
\end{eqnarray*}
This is a uniform minorization condition which immediately implies
uniform geometric convergence (see the outline of our results in Section
\ref{sec:Introduction}), but in the present situation the result
is even stronger in that, in particular, it provides us with quantitative
bounds on the dependence of the performance of the algorithm on $N$.
Indeed it is a standard result that the minorization constant
\[
\epsilon_{N}=\frac{N-1}{2\bar{G}+N-2}=1-\frac{2\bar{G}-1}{2\bar{G}+N-2}\quad,
\]
provides the upper bound $1-\epsilon_{N}$ on the (geometric) rate
of convergence of the algorithm, which here vanishes at an asymptotic
rate $N^{-1}$ as $N$ increases. As we shall see the fact that the
minorization measure is the invariant distribution leads to a direct
lower bound on associated Dirichlet forms associated to $P_{N}$ which
in turn provide quantitative bounds on the spectral gap and the associated
asymptotic variance.  In the remainder of the section we generalize
the representation of $P_{N}$ in terms of the c$^{2}$SMC algorithm
and ``the estimator of the normalizing constant'' which suggests
applying Jensen's inequality as above. This requires us to consider
estimates of the resulting expectation in Section~\ref{sec:Estimates-of-the}. 

In order to proceed further it is required to define the c$^{2}$SMC
process, which is essentially similar to the cSMC process but where
conditioning is now upon two trajectories $x,y\in\mathsf{X}$. The
definition is therefore similar, but for reasons which will become
clearer below the second fixed trajectory is set to have a lineage
of the general form $\mathbf{k}:=k_{1:T}\in[N]^{T}$. We will use
below the convention that $\delta_{a,b}\bigl({\rm d}z^{1}\times{\rm d}z^{k}\bigr)$
reduces to $\delta_{a}({\rm d}z^{1})$ whenever $k=1$. The definition
of this process is similar to that of the cSMC algorithm and the distributions
involved are defined for $x,y\in\mathsf{X}$ and $\mathbf{k}\in[N]^{T}$
as follows 
\[
\mathbb{P}_{\mathbf{1},x,\mathbf{k},y}^{N}\left(Z_{1}\in{\rm d}z_{1}\right)=\delta_{x_{1},y_{1}}\bigl({\rm d}z_{1}^{1}\times{\rm d}z_{1}^{k_{1}}\bigr)\prod_{i=2,i\neq k_{1}}^{N}M_{1}({\rm d}z_{1}^{i})\quad,
\]
and for $t=2,\ldots,T-1$ (with the convention $a_{t-1}^{k,l}:=(a_{t-1}^{k},a_{t-1}^{l})$)
\begin{multline*}
\mathbb{P}_{\mathbf{1},x,\mathbf{k},y}^{N}\left(Z_{t}\in{\rm d}z_{t},A_{t-1}=a_{t-1}\left|Z_{t-1}=z_{t-1}\right.\right)=\delta_{x_{t},y_{t}}\bigl({\rm d}z_{t}^{1}\times{\rm d}z_{t}^{k_{t}}\bigr)\\
\times\mathbb{I}\{a_{t-1}^{1,k_{t}}=(1,k_{t-1})\}\prod_{i=2,i\neq k_{t}}^{N}\Bigg(\sum_{l=1}^{N}\frac{G_{t-1}(z_{t-1}^{l})}{\sum_{j=1}^{N}G_{t-1}(z_{t-1}^{j})}\mathbb{I}\left\{ a_{t-1}^{i}=l\right\} M_{t}(z_{t-1}^{l},{\rm d}z_{t}^{i})\Bigg)
\end{multline*}
and 
\[
\mathbb{P}_{\mathbf{1},x,\mathbf{k},y}^{N}\left(A_{T}=l\left|Z_{T}=z_{T}\right.\right)={\displaystyle {\textstyle \frac{{\displaystyle G_{T}(z_{T}^{l})}}{{\displaystyle {\textstyle \sum_{j=1}^{N}}G_{T}(z_{T}^{j})}}\quad.}}
\]
We note that although the transitions and the initial distributions
are, by the convention, well defined for $k_{t}=1$ and $x_{t}\neq y_{t}$
the distribution above will never be used in such a context. Just
as $\mathbb{P}_{\mathbf{1},x}^{N}$ is not a conditional distribution
of $\mathbb{P}^{N}\bigl(\cdot\bigr)$, the law of the SMC algorithm,
the same holds between $\mathbb{P}_{\mathbf{1},x,\mathbf{k},y}^{N}\bigl(\cdot\bigr)$
and $\mathbb{P}_{\mathbf{1},x}^{N}\bigl(\cdot\bigr)$. However we
now provide an important property relating these two probability distributions,
which together with (\ref{eq:PN_defn}) will allow us to decompose
this transition into key quantities and establish the sought minorization
condition. The proof of the following Lemma is in~\citep[Appendix~\ref{sec:Proof-of-Lemma_simplecorrespondence}]{andrieu2015uniformsupplement}.
\begin{lem}
\label{lem:simple_correspondence}For $\mathbf{i}\in\{2,\ldots,N\}^{T}$
and $x\in\mathsf{X}$,
\[
\mathbb{E}_{\mathbf{1},x}^{N}\left[I_{\mathbf{i}}\bigl(Z_{1:T},A_{1:T},S\bigr)\right]=\frac{\gamma_{T}}{N^{T}}\int_{\mathsf{X}}\pi({\rm d}y)\times\mathbb{I}\{y\in S\}\times\mathbb{E}_{\mathbf{1},x,\mathbf{i},y}^{N}\left[\frac{1}{{\displaystyle {\textstyle \prod_{t=1}^{T}\frac{1}{N}\sum_{j=1}^{N}}G_{t}(Z_{t}^{j})}}\right]\quad.
\]

\end{lem}
As we shall see the concentration properties of the ``estimator of
the normalizing constant'' plays a central role for any $z_{1:T}\in\bigl(\mathsf{Z}^{N}\bigr)^{T}$
\[
\hat{\gamma}_{T}^{N}\bigl(z_{1:T}\bigr):=\prod_{t=1}^{T}\frac{1}{N}\sum_{j=1}^{N}G_{t}(z_{t}^{j})\quad.
\]
We first obtain a uniform minorization condition for the cSMC transition
probability. This simple result establishes the expectation of $\hat{\gamma}_{T}^{N}\bigl(Z_{1:T}\bigr)$
with respect to a c$^{2}$SMC algorithm as a key quantity of interest,
and motivates the non-asymptotic analysis and bounds of Section~\ref{sec:Estimates-of-the}.
\begin{prop}
\label{prop:uniformminorizationcrude}For any $(x,S)\in\mathsf{X}\times\mathcal{B}\bigl(\mathsf{X}\bigr)$
and $N\geq2$ we have
\[
P_{N}(x,S)\geq\int_{S}\frac{\gamma_{T}\times(1-1/N)^{T}}{\mathbb{E}_{\mathbf{1},x,\mathbf{2},y}^{N}\left[\hat{\gamma}_{T}^{N}\bigl(Z_{1:T}\bigr)\right]}\pi({\rm d}y)\quad.
\]
\end{prop}
\begin{proof}
Using (\ref{eq:PN_defn}), we only keep the trajectories for which
there is no coalescence with the first trajectory, i.e., we exclude
terms such that $i_{t}=1$ for some $t\in[T]$ and obtain
\[
P_{N}(x,S)\geq\sum_{\mathbf{i}\in\{2,\ldots,N\}^{T}}\mathbb{E}_{\mathbf{1},x}^{N}\left[I_{\mathbf{i}}\bigl(Z_{1:T},A_{1:T},S\bigr)\right]\quad.
\]
Consequently, using Lemma~\ref{lem:simple_correspondence},
\begin{eqnarray*}
P_{N}(x,S) & \geq & \sum_{i_{1:T}\in[2:N]^{T}}\frac{\gamma_{T}}{N^{T}}\int_{\mathsf{X}}\pi({\rm d}y)\times\mathbb{I}\{y\in S\}\times\mathbb{E}_{\mathbf{1},x,\mathbf{i},y}^{N}\left[\frac{1}{{\displaystyle {\textstyle \prod_{t=1}^{T}\frac{1}{N}\sum_{j=1}^{N}}G_{t}(Z_{t}^{j})}}\right]\\
 & = & \frac{\gamma_{T}(N-1)^{T}}{N^{T}}\int_{S}\mathbb{E}_{\mathbf{1},x,\mathbf{2},y}^{N}\left[\frac{1}{{\displaystyle {\textstyle \prod_{t=1}^{T}\frac{1}{N}\sum_{j=1}^{N}}G_{t}(Z_{t}^{j})}}\right]\pi({\rm d}y_{1:T})\quad,
\end{eqnarray*}
using invariance by permutation of $i_{1},\ldots,i_{T}$ of the expectations.
We conclude by application of Jensen's inequality for the convex function
$u\mapsto1/u$ for $u\in\mathbb{R}_{+}$. \end{proof}
\begin{cor}
\label{cor:uniformminorbymu}Let $N\geq2$ and assume that 
\[
\epsilon_{N}:=\frac{\gamma_{T}\times(1-1/N)^{T}}{\sup_{x,y\in\mathsf{X}}\mathbb{E}_{\mathbf{1},x,\mathbf{2},y}^{N}\left[\hat{\gamma}_{T}^{N}\bigl(Z_{1:T}\bigr)\right]}>0\quad,
\]
then for any $(x,S)\in\mathsf{X}\times\mathcal{B}\bigl(\mathsf{X}\bigr)$,
$P_{N}(x,S)\geq\epsilon_{N}\pi(S)$ and from Proposition~\ref{prop:uniformminorizationcrude}
all the properties of~\citep[Proposition~\ref{prop:boundsspectralgapandvarianceforP_N}]{andrieu2015uniformsupplement}
apply to the i-cSMC with $\varepsilon=\epsilon_{N}$.
\end{cor}
The next section is dedicated to finding a useful expression for the
expectation $\mathbb{E}_{\mathbf{1},x,\mathbf{2},y}^{N}\left[\hat{\gamma}_{T}^{N}\bigl(Z_{1:T}\bigr)\right]$
and establishing explicit bounds on this quantity, and therefore $\epsilon_{N}$
in Corollary~\ref{cor:uniformminorbymu}, under additional assumptions.

Before proceeding to novel analysis, for completeness we gather two
known properties of the i-cSMC (in the general set-up) in the following
lemma which will be exploited throughout the remainder of the paper.
Both results are immediate upon noticing that the i-cSMC is a two
stage Gibbs sampler on an artificial joint distribution (see (\ref{eq:definitionartificialdistribution})
in~\citep[Appendix~\ref{sec:Proof-of-Lemma_selfadjoint}]{andrieu2015uniformsupplement},
which is a generalization of (\ref{eq:artificialforiSIR})). The
results have also been shown in detail in~\citep{chopin:singh:2013}.
A proof is included in~\citep[Appendix~\ref{sec:Proof-of-Lemma_selfadjoint}]{andrieu2015uniformsupplement}
for completeness.
\begin{lem}
\label{lem:selfajoint_positive}$P_{N}$, viewed as an operator on
$L^{2}\bigl(\mathsf{X},\pi\bigr)$, is self-adjoint and positive.
\end{lem}

\section{Quantitative bounds for the doubly conditional i-cSMC expectation\label{sec:Estimates-of-the}}

In this section we first find an exact expression for $\mathbb{E}_{\mathbf{1},x,\mathbf{2},y}^{N}\left[\hat{\gamma}_{T}^{N}\bigl(Z_{1:T}\bigr)\right]$
in terms of quantities underpinning the definition of $\pi$ given
in Section~\ref{sec:The-i-CSMC} and then move on to provide various
estimates of the conditional expectation involved in the minorization
established in Proposition~\ref{prop:uniformminorizationcrude},
under various assumptions on the aforementioned quantities. Throughout
we use the usual convention that $\sum_{\emptyset}=0$ and $\prod_{\emptyset}=1$.
We let $G_{p,q}(z):=Q_{p,q}(1)(z)$ and $G_{p,q}^{1+2}:=G_{p,q}\bigl(x_{p}\bigr)+G_{p,q}\bigl(y_{p}\bigr)$.
We note that $G_{p,p+1}(z)=G_{p}(z)$ for $p\in[T]$ and we use the
convention throughout that for any $z\in\mathsf{Z}$, $G_{0}(z)=1$
and $Q_{0,p}\bigl(f_{p}\bigr)(z):=M_{1}\left(Q_{1,p}(f_{q})\right)$.
We write $G_{0,p}:=M_{1}\left(Q_{1,p}(1)\right)$ since $G_{0,p}(z)$
is independent of $z$. Our first result, whose proof can be found
in~\citep[Appendix~\ref{sec:Supplementary-quantitative}]{andrieu2015uniformsupplement},
is
\begin{prop}
\label{prop:doubleconditionalSMC}Let $x,y\in\mathsf{X}$ and $N\geq2$.
Then, 
\[
\mathbb{E}_{\mathbf{1},x,\mathbf{2},y}^{N}\left[\hat{\gamma}_{T}^{N}\bigl(Z_{1:T}\bigr)\right]=\frac{1}{N^{T}}\sum_{s=1}^{T+1}(N-2)^{T+1-s}\sum_{\mathbf{i}\in\mathcal{I}_{T+1,s}}G_{0,i_{\text{1}}}C_{T,s}\bigl(\mathbf{i},x,y\bigr)\quad,
\]
where for any $s=1,\ldots,k$, 
\[
\mathcal{I}_{k,s}:=\left\{ i_{1},\ldots,i_{s}\in\mathbb{N}^{s}:T-k+1<i_{1}\cdots<i_{s}=T+1\right\} \quad,
\]
and for $\mathbf{i}\in\mathcal{I}_{k,s}$ 
\[
C_{k,s}\bigl(\mathbf{i},x,y\bigr):=\prod_{m=1}^{s-1}\bigl[G_{i_{m},i_{m+1}}\bigl(x_{i_{m}}\bigr)+G_{i_{m},i_{m+1}}\bigl(y_{i_{m}})\bigr]\quad.
\]
\end{prop}
\begin{rem}
While the expectation of interest here has been hitherto uninvestigated,
the form of Proposition~\ref{prop:doubleconditionalSMC} is reminiscent
of non-asymptotic results in~\citep{cerou-delmoral-guyader}, in
which second moments of $\hat{\gamma}_{T}^{N}\bigl(Z_{1:T}\bigr)$
are analyzed with respect to the law of a standard SMC algorithm.
\end{rem}
We now turn to estimates of the expectation above, starting with very
minimal assumptions which allow us to establish the minorization condition
required to apply~\citep[Proposition~\ref{prop:boundsspectralgapandvarianceforP_N}]{andrieu2015uniformsupplement}
and deduce most of our results, without the need for assumptions on
the dynamic of the system---the number of particles is however required
to grow exponentially in order to maintain a set level of performance.
We show subsequently that with stronger assumptions on $\{M_{t},G_{t}\}_{t=1}^{T}$
it is possible to show that $N$ should grow linearly with $T$ to
ensure that a set level of performance is maintained.
\begin{prop}
\label{prop:boundepsilonNwithuniformboundG}Assume that for all $t\in\{1,\ldots,T\}$,
$\bar{G}_{t}:=\sup_{z\in\mathsf{Z}}G_{t}(z)<\infty$, then for any
$N\geq2$
\[
\mathbb{E}_{\mathbf{1},x,\mathbf{2},y}^{N}\left[\hat{\gamma}_{T}^{N}\bigl(Z_{1:T}\bigr)\right]\leq\gamma_{T}\left\{ 1+\left[1-\left(1-\frac{2}{N}\right)^{T}\right]\left[\frac{\prod_{t=1}^{T}\bar{G}_{t}}{\gamma_{T}}-1\right]\right\} \quad.
\]
\end{prop}
\begin{proof}
The assumption on the potentials implies that for any $p,q\in\mathbb{N}$
with $p<q$ we have $G_{p,q}\leq\prod_{k=p}^{q-1}\bar{G}_{k}$, and
from Proposition~\ref{prop:doubleconditionalSMC} we have
\begin{eqnarray*}
\mathbb{E}_{\mathbf{1},x,\mathbf{2},y}^{N}\left[\hat{\gamma}_{T}^{N}\bigl(Z_{1:T}\bigr)\right] & = & \sum_{s=1}^{T+1}\left(\frac{N-2}{N}\right)^{T+1-s}\frac{2^{s-1}}{N{}^{s-1}}\sum_{\mathcal{I}_{T+1,s}}G_{0,i_{\text{1}}}\prod_{m=1}^{s-1}\frac{1}{2}G_{i_{m},i_{m+1}}^{1+2}\\
 & \leq & \gamma_{T}\left(\frac{N-2}{N}\right)^{T}+\prod_{k=1}^{T}\bar{G}_{k}\times\sum_{s=2}^{T+1}\binom{T}{s-1}\left(\frac{N-2}{N}\right)^{T+1-s}\frac{2^{s-1}}{N{}^{s-1}}\\
 & = & \gamma_{T}\left(\frac{N-2}{N}\right)^{T}+\left[1-\left(\frac{N-2}{N}\right)^{T}\right]\prod_{k=1}^{T}\bar{G}_{k}\quad,
\end{eqnarray*}
and the result follows.\end{proof}
\begin{cor}
\label{cor:convergence_epsilon}Propositions~\ref{prop:uniformminorizationcrude}
and~\ref{prop:boundepsilonNwithuniformboundG} together imply that
for any $x,S\in\mathsf{X}\times\mathcal{B}(\mathsf{X})$,
\[
P_{N}(x,S)\geq\epsilon_{N}\pi(S)\quad\text{with}\quad\epsilon_{N}=\frac{(1-1/N)^{T}}{1+\left[1-\left(1-\frac{2}{N}\right)^{T}\right]\left[\frac{\prod_{t=1}^{T}\bar{G}_{t}}{\gamma_{T}}-1\right]}\quad
\]
 and $\lim_{N\rightarrow\infty}\epsilon_{N}=1$. 
\end{cor}
It should be clear that despite Corollary~\ref{cor:convergence_epsilon},
the term $\prod_{t=1}^{T}\bar{G}_{t}/\gamma_{T}$ typically grows
exponentially fast with $T$ whenever the potentials are not constant
functions. Therefore, Proposition~\ref{prop:boundepsilonNwithuniformboundG}
suggests that the number of particles $N$ should grow exponentially
with $T$ in general. However, stronger assumptions on the system
under consideration will allow us to maintain a given lower bound
on $\epsilon_{N}$ by increasing $N$ only linearly with $T$. We
first state our main result using the abstract condition (A\ref{hyp:mixingabstract})
and then show that classical strong mixing conditions (A\ref{hyp:strongmixingpotentialassumptions})
imply (A\ref{hyp:mixingabstract}). 
\begin{prop}
\label{prop:mixing_bound}Assume (A\ref{hyp:mixingabstract}), then
for any $N\geq2$
\[
\mathbb{E}_{\mathbf{1},x,\mathbf{2},y}^{N}\left[\hat{\gamma}_{T}^{N}\bigl(Z_{1:T}\bigr)\right]\leq\gamma_{T}\left(1+\frac{2(\alpha-1)}{N}\right)^{T}\quad.
\]
\end{prop}
\begin{proof}
First notice that for any $1\leq k\leq n$
\[
Q_{0,n}(1)=Q_{0,k}(1)\frac{Q_{0,n}(1)}{Q_{0,k}(1)}=Q_{0,k}(1)\eta_{k}Q_{k,n}(1)\quad,
\]
and therefore for any $s\in\{1,\ldots,T\}$ and $0<i_{1}<\cdots<i_{s-1}<i_{s}=T+1$
with the notation defined earlier,
\[
Q_{0,T}(1)=Q_{0,i_{1}}(1)\prod_{k=1}^{s-1}\eta_{i_{k}}Q_{i_{k},i_{k+1}}(1)=G_{0,i_{1}}\prod_{k=1}^{s-1}\eta_{i_{k}}G_{i_{k},i_{k+1}}\quad,
\]
and from (A\ref{hyp:mixingabstract}), with $\bar{G}_{p,q}:=\sup_{z\in\mathsf{Z}}G_{p,q}(z)$,
and applying Proposition \ref{prop:doubleconditionalSMC} yields the
following upper bound for $\mathbb{E}_{\mathbf{1},x,\mathbf{2},y}^{N}\left[\hat{\gamma}_{T}^{N}\bigl(Z_{1:T}\bigr)\right]$:
\begin{eqnarray*}
 & {\displaystyle \sum_{s=1}^{T+1}} & \left(\frac{N-2}{N}\right)^{T+1-s}\frac{2^{s-1}}{N{}^{s-1}}\sum_{0<i_{1}<\cdots<i_{s-1}<i_{s}=T+1}G_{0,i_{\text{1}}}\prod_{m=1}^{s-1}\frac{1}{2}G_{i_{m},i_{m+1}}^{1+2}\\
 & \leq & \gamma_{T}\left(\frac{N-2}{N}\right)^{T}+\gamma_{T}\sum_{s=2}^{T+1}\left(\frac{N-2}{N}\right)^{T+1-s}\frac{2^{s-1}}{N{}^{s-1}}\sum_{\mathcal{I}_{T+1,s}}\frac{G_{0,i_{\text{1}}}}{G_{0,i_{1}}}\prod_{m=1}^{s-1}\frac{\bar{G}_{i_{m},i_{m+1}}}{\eta_{i_{k}}G_{i_{k},i_{k+1}}}\\
 & \le & \gamma_{T}\sum_{s=1}^{T+1}\binom{T}{s-1}\left(\frac{N-2}{N}\right)^{T+1-s}\frac{2^{s-1}}{N{}^{s-1}}\alpha^{s-1}\quad,
\end{eqnarray*}
and we conclude by an application of the binomial theorem.\end{proof}
\begin{cor}
\label{cor:linear_in_T_epsilon_N_bound}Propositions~\ref{prop:uniformminorizationcrude}
and~\ref{prop:boundepsilonNwithuniformboundG} together imply that
for any $(x,S)\in\mathsf{X}\times\mathcal{B}(\mathsf{X})$,
\[
P_{N}(x,S)\geq\epsilon_{N}\pi(S)\quad\text{with}\quad\epsilon_{N}=\left(\frac{1-1/N}{1+\frac{2(\alpha-1)}{N}}\right)^{T}\quad.
\]
Now, let $N-1\geq CT$ for some $C>0$. Then $\epsilon_{N}\geq\exp\left(-\frac{2\alpha-1}{C}\right)$.\end{cor}
\begin{proof}
Propositions~\ref{prop:uniformminorizationcrude} and~\ref{prop:mixing_bound}
together imply that 
\[
\epsilon_{N}\geq\left(1+\frac{2\alpha-1}{N-1}\right)^{-T}.
\]
Since $(N-1)\geq CT$ for some $C>0$, and $\log(1+x)\leq x$ for
all $x\geq0$,
\[
\left(1+\frac{2\alpha-1}{N-1}\right)^{T}\le\left(1+\frac{2\alpha-1}{CT}\right)^{T}\leq\exp\left(\frac{2\alpha-1}{C}\right)\quad.\hfill\qedhere
\]
\end{proof}
\begin{rem}
\label{rem:tuning}The combination of the upper bound of ${\rm var}(f,P_{N})$
in Theorem~\ref{thm:THEtheorem} with Corollary~\ref{cor:linear_in_T_epsilon_N_bound}
suggests a rough rule of thumb to select $N$ for the i-cSMC Markov
kernel. In particular, there is generally a tradeoff between iterating
a less computationally intensive Markov kernel more times and iterating
a more computationally intensive expensive fewer times. This suggests
that one should minimize the function $f(N):=N{\rm var}(f,P_{N})$.
While an analytic expression for ${\rm var}(f,P_{N})$ is not available
we can minimize its upper bound 
\[
(CT+1)\left\{ 2\exp\left(\frac{2\alpha-1}{C}\right)-1\right\} ,
\]
with respect to $C$. Assuming that we are in the scenario where $N\gg1$
and therefore $CT+1\approx CT$ one then finds the unique minimum
\[
C^{*}=\frac{2\alpha-1}{{\rm LambertW}(-\frac{1}{2\exp(1)})+1}\approx1.302\left(2\alpha-1\right)\;,
\]
(where\textcolor{black}{{} ${\rm Lambert_{W}}$ is the principal branch
of the Lambert W }function) or correspondingly
\[
\epsilon_{N}^{*}\approx0.464\,.
\]

\end{rem}
Hence, under (A\ref{hyp:mixingabstract}) it is only required for
$N$ to scale linearly with $T$ in order to maintain a non-vanishing
ergodicity rate. Following, e.g.,~\citep{delmoral:2004,cerou-delmoral-guyader}
we make the following assumptions on $\{M_{t}\}$ and the potentials
$\{G_{t}\}$ which combined define an $m$-step ``strong mixing''
condition which automatically implies (A\ref{hyp:mixingabstract}).
The following result relies on classical arguments~\citep[Lemma 4.3]{delmoral:2004,cerou-delmoral-guyader}
\begin{lem}
\label{lem:A2impliesA1}Assume (A\ref{hyp:strongmixingpotentialassumptions}).
Then for any $k\in\mathbb{Z}_{+}$ we have
\[
\sup_{z,z'\in\mathsf{Z}^{2}}\frac{Q_{p,p+k}(1)(z)}{Q_{p,p+k}(1)(z')}\leq\beta\delta\quad,
\]
i.e., (A\ref{hyp:mixingabstract}) is satisfied.
\end{lem}

\section{Necessity of the boundedness assumption and a conjecture\label{sec:conjectureonboundedness}}

Proposition~\ref{prop:boundepsilonNwithuniformboundG} showed that
the i-cSMC kernel is uniformly ergodic if the potentials are bounded.
We study here the opposite case, where at least one of the potentials
is unbounded. We discover that then the algorithm cannot be uniformly
ergodic (Proposition~\ref{prop:non-uniform}), and in many cases
the algorithm cannot be geometrically ergodic (Proposition~\ref{prop:non-geom-easy}
and Lemma~\ref{lem:full-supp}; Remark~\ref{rem:strictly_positive}).
We believe that the latter holds in general (Conjecture~\ref{conj:general-necessity}),
but a proof has remained elusive. This dichotomy of algorithms which
are uniformly ergodic and sub-geometrically ergodic would be in perfect
analogy with the behaviour of the independent Metropolis--Hastings~\citep[~Theorem 2.1]{mengersen-tweedie}.

We will denote hereafter the marginal densities of $\pi$ by 
\begin{equation}
\pi_{t:u}(A):=\pi(\mathsf{Z}^{t-1}\times A\times\mathsf{Z}^{T-u})\qquad\text{for \ensuremath{A\in\mathcal{B}(\mathsf{Z}^{u-t+1})},}\label{eq:definitionpi_t:u}
\end{equation}
where $1\le t\le u\le T$ and we use the shorthand $\pi_{t}(A):=\pi_{t:t}(A)$.

In this section, we will assume that $\mathsf{S}\in\mathcal{B}(\mathsf{Z})^{T}$
is a fixed set such that for all $x\in\mathsf{S}$, $\prod_{t=1}^{T}G_{t}(x_{t})>0$
and $\pi(\mathsf{S})=1$. Further, $\mathsf{S}$ contains all possible
starting points of the algorithm, that is, we assume that the state
space of the i-cSMC is $\mathsf{S}$. In the discrete case, the minimal
$\mathsf{S}$ consists of the points of positive $\pi$-measure, and
in the continuous case where $\pi$ admits a density, the set $\mathsf{S}$
can be taken as the set where the density is positive.

Further, we will assume that $\pi_{1}$ is not concentrated on a single
point. We can do this without loss of generality, because if $\pi_{1},\ldots,\pi_{t}$
were concentrated on single points of the state space, the algorithm
would be deterministic until $\pi_{t+1}$ and we could consider the
i-cSMC for $\pi'=\pi_{t+1:T}$.
\begin{prop}
\label{prop:non-uniform} Suppose $\pi_{t}$-${\rm ess}\sup_{x_{t}}G_{t}(x_{t})=\infty$
for some $t\in[T]$. Then, the i-cSMC kernel $P_{N}$ is not uniformly
ergodic for any $N\in\mathbb{N}$. \end{prop}
\begin{proof}
If the i-cSMC kernel is uniformly ergodic, then there exist $K<\infty$
and $\rho\in(0,1)$ such that 
\[
\sup_{x\in\mathsf{S}}\|P_{N}^{n}(x,\,\cdot\,)-\pi(\,\cdot\,)\|_{TV}\le K\rho^{n}\qquad\text{for all \ensuremath{n\in\mathbb{N}}.}
\]
Fix $\epsilon'>0$ and let $n\in\mathbb{N}$ be such that $K\rho^{n}\le\epsilon'$.
We will prove that there exists a set $B_{\epsilon'}\in\mathcal{B}(\mathsf{Z})$
such that $\pi_{1}(B_{\epsilon'})>0$ and $\inf_{x\in B_{\epsilon'}}P_{N}^{n}(x,\{x_{1}\}\times\mathsf{Z}^{T-1})\ge1-\epsilon'$.
For all $x\in B_{\epsilon'}$, we have $|P_{N}^{n}(x,\{x_{1}\}\times\mathsf{Z}^{T-1})-\pi_{1}(\{x_{1}\})|\le K\rho^{n}\le\epsilon'$.
This, with $\epsilon'>0$ small enough, will contradict $\pi_{1}(\{x_{1}\})<1$.

Lemma~\ref{lem:choose_c1c2} shows that there exists $\phi:\mathbb{R}_{+}\to\mathbb{R}_{+}$
such that $\lim_{g\to\infty}\phi(g)=0$, and 
\[
P_{N}(x,\{x_{1}\}^{\complement}\times\mathsf{Z}^{T-1})\le\phi(G(x_{t})).
\]
Denote the level set $L_{t}(\underline{G}):=\{x_{t}\in\mathsf{Z}\,:\,G_{t}(x_{t})\le\underline{G}\}$.
Lemma~\ref{lem:choose_c1c2} shows that there exists $c_{2}=c_{2}(N)\in[1,\infty)$
such that for $G_{t}(x_{t})\ge\underline{G}$ 
\[
P_{N}(x,\mathsf{Z}^{t-1}\times L_{t}(\underline{G})\times\mathsf{Z}^{T-t})\le c_{2}\underline{G}/G_{t}(x_{t}).
\]
Let $\epsilon\in(0,1)$ and define $\delta:=\epsilon/c_{2}$ and let
$G_{*}$ be large enough so that $\phi(\delta^{n}G_{*})\le\epsilon$.
Define the (sub-probability) kernels $\mu_{\bar{G}}(x,{\rm d}y):=P_{N}(x,{\rm d}y)\delta_{x_{1}}(y_{1})\mathbb{I}\left\{ G_{t}(y_{t})\ge\bar{G}\right\} $
on $(\mathsf{S},\mathcal{B}(\mathsf{S}))$ for any $\bar{G}>0$ and
observe that we may estimate 
\begin{align*}
 & \mathbb{I}\left\{ G_{t}(x_{t})\ge G_{*}\right\} P_{N}^{n}(x,\{x_{1}\}\times\mathsf{Z}^{T-1})\\
 & \ge\mathbb{I}\left\{ G_{t}(x_{t})\ge G_{*}\right\} \int\mu_{\delta G_{*}}(x,{\rm d}y^{(2)})\int\mu_{\delta^{2}G_{*}}(y^{(2)},{\rm d}y^{(3)})\cdots\int\mu_{\delta^{n-1}G_{*}}(y^{(n-1)},{\rm d}y^{(n)}).
\end{align*}
We may estimate for any $i\in[n]$ and all $x\in\mathsf{S}$ such
that $G_{t}(x_{t})\ge\delta^{i-1}G_{*}$, 
\begin{align*}
\int\mu_{\delta^{i}G_{*}}(x,{\rm d}y) & \ge1-P_{N}(x,\{x_{1}\}^{\complement}\times\mathsf{Z}^{T-1})-P_{N}(x,\mathsf{Z}^{t-1}\times L_{t}(\delta^{i}G_{*})\times\mathsf{Z}^{T-t})\\
 & \ge1-2\epsilon.
\end{align*}
We conclude that for $x\in\mathsf{S}$ such that $G_{t}(x_{t})\ge G_{*}$,
\[
P_{N}^{n}(x,\{x_{1}\}\times\mathsf{Z}^{T-1})\ge(1-2\epsilon)^{n}.
\]
This proves the claim, as $\epsilon>0$ was arbitrary.\end{proof}
\begin{lem}
\label{lem:choose_c1c2}For all $x\in\mathsf{S}$ and all $\underline{G}\in\mathbb{R}_{+}$,
\begin{enumerate}
\item $P_{N}(x,\{x_{1}\}^{\complement}\times\mathsf{Z}^{T-1})\le\phi(G(x_{t}))$,
\item $P_{N}(x,\mathsf{Z}^{t-1}\times L_{t}(\underline{G})\times\mathsf{Z}^{T-t})\le(N-1)^{2}\underline{G}/G_{t}(x_{t})\quad\text{whenever \ensuremath{G_{t}(x_{t})\ge\underline{G}}}$.
\end{enumerate}
where $\phi:\mathbb{R}_{+}\to\mathbb{R}_{+}$ is a function such that
$\lim_{g\to\infty}\phi(g)=0$.\end{lem}
\begin{proof}
In both cases, we consider the case $t<T$; the special case $t=T$
can be treated similarly. In order to facilitate the theoretical analysis,
we introduce a non-standard implementation of the cSMC which relies
on the remark that at any time instant a given particle can only have
a maximum number $N$ of children. Hence when implementing the cSMC
it is always possible to draw $N$ children first and then decide
who is carried forward according to the standard selection mechanism.
It is in fact possible to push this idea further and, given a fixed
$x\in\mathsf{S}$, to sample the following $N$-ary tree of random
variables first 
\begin{align*}
\hat{Z}_{1}^{1} & =x_{1}, & \hat{Z}_{1}^{i} & \sim M_{1}(\cdot), & i\in[N]\setminus\{1\}\\
\hat{Z}_{2}^{1,1} & =x_{2}, & \hat{Z}_{2}^{i,j} & \sim M_{2}(\hat{Z}_{1}^{i},\cdot), & (i,j)\in[N]^{2}\setminus\{(1,1)\}\\
 &  &  & \vdots\\
\hat{Z}_{T}^{{\bf 1}} & =x_{T}, & \hat{Z}_{T}^{i_{1},\ldots,i_{T}} & \sim M_{T}(\hat{Z}_{T-1}^{i_{1},\ldots,i_{T-1}},\cdot), & (i_{1},\ldots,i_{T})\in[N]^{T}\setminus\{{\bf 1}\}\;,
\end{align*}
and then prune the tree using the selection mechanism of the cSMC
algorithm with fixed path $x\in\mathsf{S}$. As a result, each $Z_{t}^{j}$
in the cSMC is associated with some $\hat{Z}_{t}^{i}$. The construction
above permits the bound
\[
U:=\sum_{{\bf i}\in[N]^{t}}G_{t}(Z_{t}^{i_{t}})\mathbb{I}\left\{ i_{1}\neq1\right\} \prod_{p=2}^{t}\mathbb{I}\left\{ i_{p-1}=A_{p-1}^{i_{p}}\right\} \leq\sum_{{\bf i}\in\{2,\ldots,N-1\}^{t}}G_{t}(\hat{Z}_{t}^{{\bf i}})=:V,
\]
where $U$ corresponds to the sum of potentials associated with those
$Z_{t}^{j}$ whose ancestral lineage does not contain the value $1$.
It therefore follows that 
\begin{align*}
P_{N}(x,\{x_{1}\}^{\complement}\times\mathsf{Z^{T-1})}=\mathbb{E}_{{\bf 1},x}^{N}\left[\frac{U}{G_{t}(x_{t})+\sum_{j=2}^{N}G_{t}(Z_{t}^{j})}\right] & \leq\mathbb{E}_{{\bf 1},x}^{N}\left[\frac{U}{G_{t}(x_{t})+U}\right]\\
 & \le\mathbb{E}_{{\bf 1},x}^{N}\left[\frac{V}{G_{t}(x_{t})+V}\right],
\end{align*}
 because $u\mapsto u/(g+u)$ is increasing. Now, $V$ is a finite
non-negative random variable independent of $x$. We may define
\[
\phi(g):=\mathbb{E}_{{\bf 1},x}^{N}\left[\frac{V}{g+V}\right]
\]
which satisfies $\lim_{g\to\infty}\phi(g)=0$ by the monotone convergence
theorem.

For the second inequality, we can show similarly that for $G_{t}(x_{t})\ge\underline{G}$
\[
\mathbb{P}_{{\bf 1},x}^{N}\left[G_{t}(Z_{t}^{A_{t}^{i}})\leq\underline{G}\right]=\mathbb{E}_{{\bf 1},x}^{N}\left[\frac{\sum_{k=2}^{N}G_{t}(Z_{t}^{k})\mathbb{I}\left\{ G_{t}(Z_{t}^{k})\leq\underline{G}\right\} }{G_{t}(x_{t})+\sum_{k=2}^{N}G_{t}(Z_{t}^{k})}\right]\leq\frac{(N-1)\underline{G}}{G_{t}(x_{t})}
\]
and so 
\[
P_{N}(x,\mathsf{Z}^{t-1}\!\times L_{t}(\underline{G})\times\mathsf{Z}^{T-t})\!\le\!\sum_{i=2}^{N}\mathbb{P}_{{\bf 1},x}^{N}\left[G_{t}(Z_{t}^{A_{t}^{i}})\leq\underline{G}\right]\!\!=\!(N-1)\mathbb{P}_{{\bf 1},x}^{N}\left[G_{t}(Z_{t}^{A_{t}^{i}})\leq\underline{G}\right].\qedhere
\]

\end{proof}
To establish that $P_{N}$ cannot be even geometrically ergodic whenever
$\pi_{t}$-${\rm ess}\sup_{x_{t}}G_{t}(x_{t})=\infty$ for some $t\in[T]$
in many settings, we use Proposition~\ref{prop:sticky-non-geometric}.
This allows for the developments of Proposition~\ref{prop:non-geom-easy}
and Lemma~\ref{lem:full-supp}, leading to the desired result under
assumptions satisfied in many applications; see Remark~\ref{rem:strictly_positive}.
\begin{prop}
\label{prop:sticky-non-geometric}Suppose $P$ is an ergodic Markov
kernel on a state space $\big(\mathsf{X},\mathcal{B}(\mathsf{X})\big)$
with invariant distribution $\pi$. Suppose that for any $\epsilon,\delta>0$
there exists a set $A\in\mathcal{B}(\mathsf{X})$ such that $\pi(A)\in(0,\delta)$
and $\inf_{x\in A}P(x,A)\ge1-\epsilon$. Then $P$ is not geometrically
ergodic. \end{prop}
\begin{proof}
The result follows directly by following the proof of~\citep[~Theorem 3.1]{roberts-tweedie},
or by a conductance argument~\citep[~Theorem 1]{lee-latuszynski}.\end{proof}
\begin{prop}
\label{prop:non-geom-easy}Assume that for at least one $t\in[T]$
\begin{equation}
\pi\text{-}{\rm ess}\sup_{x}\,\mathbb{E}_{{\bf 1},x}^{N}\bigg[\frac{G_{t}(x_{t})}{\sum_{k=1}^{N}G_{t}(Z_{t}^{k})}\bigg]=1.\label{eq:suff}
\end{equation}
Then $P_{N}$ cannot be geometrically ergodic.\end{prop}
\begin{proof}
Because of Proposition~\ref{prop:sticky-non-geometric} it suffices
to establish that 
\begin{equation}
\pi_{1:t}\text{-}{\rm ess}\sup_{x_{1:t}\in\mathsf{S}}\bigg\{\inf_{x_{t+1:T}}P_{N}\left(x_{1:T};\{x_{1:t}\}\times\mathsf{Z}^{T-t}\right)\bigg\}=1.\label{eq:sticky-simple}
\end{equation}
We note that 
\begin{eqnarray*}
\inf_{x_{t+1:T}}P_{N}\left(x_{1:T},\{x_{1:t}\}\times\mathsf{Z}^{T-t}\right) & \geq & \mathbb{P}_{{\bf 1},x}^{N}\left(A_{t}^{1:N}=1\right)\\
 & \ge & 1-\sum_{i=2}^{N}\mathbb{P}_{{\bf 1},x}^{N}\left(A_{t}^{i}\neq1\right),
\end{eqnarray*}
because $A_{t}^{1}=1$ by construction. We emphasize that $A_{t}^{i}$
are independent of $x_{t+1:T}$. Now (\ref{eq:sticky-simple}) follows
directly from (\ref{eq:suff}) because for $i\in\{2,\ldots,N\}$,
\[
\mathbb{P}_{{\bf 1},x}^{N}\left(A_{t}^{i}=1\right)=\mathbb{E}_{{\bf 1},x}^{N}\left[\frac{G_{t}(x_{t})}{\sum_{j=1}^{N}G_{t}(Z_{t}^{j})}\right].\qedhere
\]
\end{proof}
\begin{lem}
\label{lem:equiv-easy-condition}Assume that for any $\epsilon>0$
\[
\pi\text{-}{\rm ess}\inf_{x}\,\mathbb{P}_{{\bf 1},x}^{N}\bigg(\frac{G_{t}(Z_{t}^{2})}{G_{t}(x_{t})}\ge\epsilon\bigg)=0.
\]
Then, (\ref{eq:suff}) holds.\end{lem}
\begin{proof}
For any $\epsilon,\delta>0$ there exists $A_{\epsilon,\delta}$ such
that $\pi(A_{\epsilon,\delta})>0$ and for $x\in A_{\epsilon,\delta}$
\[
\mathbb{P}_{{\bf 1},x}^{N}\bigg(\frac{G_{t}(Z_{t}^{2})}{G_{t}(x_{t})}\ge\epsilon\bigg)<\delta.
\]
Because of exchangeability, for any $x$ and $2\le k\le N$, 
\[
\mathbb{P}_{{\bf 1},x}^{N}\bigg(\frac{G_{t}(Z_{t}^{k})}{G_{t}(x_{t})}\ge\epsilon\bigg)=\mathbb{P}_{{\bf 1},x}^{N}\bigg(\frac{G_{t}(Z_{t}^{2})}{G_{t}(x_{t})}\ge\epsilon\bigg).
\]
Denote $B=\big\{\frac{\sum_{k=2}^{N}G_{t}(Z_{t}^{k})}{G_{t}(x_{t})}\ge(N-1)\epsilon\big\}$,
then for $x\in A_{\epsilon,\delta}$ also 
\[
\mathbb{P}_{{\bf 1},x}^{N}(B)\le\sum_{k=2}^{N}\mathbb{P}_{{\bf 1},x}^{N}\bigg(\frac{G_{t}(Z_{t}^{k})}{G_{t}(x_{t})}\ge\epsilon\bigg)<(N-1)\delta.
\]
We may bound for any $x\in A_{\epsilon,\delta}$, 
\begin{align*}
\mathbb{E}_{{\bf 1},x}^{N}\bigg[\frac{G_{t}(x_{t})}{\sum_{k=1}^{N}G_{t}(Z_{t}^{k})}\bigg] & \ge\mathbb{E}_{{\bf 1},x}^{N}\bigg[\mathbb{I}\left\{ B^{\complement}\right\} \frac{G_{t}(x_{t})}{\sum_{k=1}^{N}G_{t}(Z_{t}^{k})}\bigg]\\
 & \ge\mathbb{E}_{{\bf 1},x}^{N}\bigg[\left\{ B^{\complement}\right\} \frac{1}{1+(N-1)\epsilon}\bigg]\\
 & \ge\frac{1-(N-1)\delta}{1+(N-1)\epsilon}.
\end{align*}
Letting $\epsilon,\delta\to0$ completes the proof.\end{proof}
\begin{lem}
\label{lem:full-supp}Assume that there exists $t\in[T]$ such that
$\pi_{t}$-${\rm ess}\sup_{x_{t}}G_{t}(x_{t})=\infty$, and if $t\ge2$,
suppose also that for any $A\in\mathcal{B}(\mathsf{Z}^{1:t-1})$ and
$B\in\mathcal{B}(\mathsf{Z})$, 
\[
\pi_{1:t-1}(A)>0\text{ and }\pi_{t}(B)>0\implies\pi_{1:t}(A\times B)>0.
\]
Then, the assumption of Lemma~\ref{lem:equiv-easy-condition} and
consequently (\ref{eq:suff}) holds for $t$.\end{lem}
\begin{proof}
Assume that $t\in\{2,\ldots,T\}$, and for any $x_{1:t-1}\in\mathsf{Z}^{t-1}$
let $\mu_{x_{1:t-1}}$ denote the distribution of $G_{t}(Z_{t}^{2})$
under $\mathbb{P}_{{\bf 1},x_{1:t-1}}^{N}$. By~\citep[Lemma~\ref{lem:tightness}]{andrieu2015uniformsupplement},
there exists $A\in\mathcal{B}(\mathsf{Z}^{t-1})$ such that $\pi_{1:t-1}(A)\ge1/2$
and the family $\{\mu_{x_{1:t-1}}\}_{x_{1:t-1}\in A}$ is tight. Therefore,
for any $\epsilon,\delta>0$ there exists $\bar{G}_{t}<\infty$ such
that $\mathbb{P}_{{\bf 1},x}^{N}(G_{t}(Z_{t}^{2})/\bar{G}_{t}\ge\epsilon)<\delta$
for all $x_{1:t-1}\in A$. Because $\pi_{t}$-${\rm ess}\sup_{x_{t}}G_{t}(x_{t})=\infty$,
the set $A\times\{x_{t}\,:\,G_{t}(x_{t})\ge\bar{G}_{t}\}\times Z^{T-t-1}$
is of positive $\pi$-measure. The case $t=1$ follows similarly because
the distribution of $G_{1}(Z_{1}^{2})$ is independent of $x$.\end{proof}
\begin{rem}
\label{rem:strictly_positive}An immediate implication of Propositions~\ref{prop:non-geom-easy}
and~\ref{prop:boundepsilonNwithuniformboundG} and Lemma~\ref{lem:full-supp}
is that if $\pi$ is equivalent to a Lebesgue or counting measure
on $\mathsf{X}$ then $P_{N}$ is geometrically ergodic for any $N\geq2$
if and only if $\pi_{t}$-${\rm ess}\sup_{x_{t}}G_{t}(x_{t})<\infty$
for all $t\in[T]$. This covers many applications in statistics, where
often the potentials $G_{t}$ are strictly positive and for any $x_{t}\in\mathsf{Z}$,
the Markov kernel $M_{t}(x_{t},\cdot)$ is equivalent to a Lebesgue
or counting measure on $\mathsf{Z}$.
\end{rem}
Proposition~\ref{prop:non-geom-easy} does not characterize all situations
in which $P_{N}$ fails to be geometrically ergodic. Indeed, in the
following example (\ref{eq:suff}) does not hold, and $P_{N}$ still
fails to be geometrically ergodic.
\begin{example*}
\label{ex:not-satisfying}Let $\mathsf{Z}=\mathbb{N}$, $T=2$, $G_{1}(z)\equiv1$
and $M_{1}(z_{1})$ be any probability distribution supported on $\mathbb{N}$
(e.g.,~a Poisson distribution). Define $M_{2}(z_{1},z_{2})=\frac{1}{2}\delta_{2z_{1}}(z_{2})+\frac{1}{2}\delta_{2z_{1}+1}(z_{2})$
and $G_{2}(z_{2})=z_{2}$. It is not difficult to see that this example
does not satisfy (\ref{eq:suff}), but $\pi_{2}$-${\rm ess}\sup_{z_{2}}G_{2}(z_{2})=\infty$.
It is easy to observe as well that the sets $A_{n}:=\{(n,2n),(n,2n+1)\}$
satisfy $\pi(A_{n})>0$ and that $\inf_{x\in A}P_{N}(x,A_{n})\ge1-\delta_{n}$
where $\delta_{n}\to0$ as $n\to\infty$. 
\end{example*}
Our findings above suggest that the essential boundedness of the potentials
could in fact be a necessary condition for geometric ergodicity. We
have considered also various other examples, and it seems that in
any specific scenario it is easy to identify ``sticky'' sets and
conclude by Lemma~\ref{prop:sticky-non-geometric}. However, we have
yet to identify such sets in general, and so have resorted to stating
the following.
\begin{conjecture}
\label{conj:general-necessity}Suppose $\pi_{t}$-${\rm ess}\sup_{x_{t}}G_{t}(x_{t})=\infty$
for some $t\in[T]$. Then, the i-cSMC kernel is not geometrically
ergodic for any $N\in\mathbb{N}$.
\end{conjecture}

\section{\label{sec:The-particle-Gibbs}The particle Gibbs sampler}

In numerous situations of practical interest one is interested in
sampling from a probability distribution $\pi\bigl({\rm d}\theta\times{\rm d}x\bigr)$
defined on some measurable space $\bigl(\Theta\times\mathsf{X},\mathcal{B}(\Theta)\times\mathsf{\mathcal{B}(X)}\bigr)$
for which direct sampling is difficult, but sampling from the associated
conditional probability distributions $\pi_{\theta}({\rm d}x)$ and
$\pi_{x}({\rm d}\theta)$ for any $(\theta,x)\in\Theta\times\mathsf{X}$
turns out to be easier. In fact when sampling exactly from these conditionals
is possible one can define the two stage Gibbs sampler~\citep{robert1999monte}
which alternately samples from these conditional distributions. More
precisely, let us define, for any $(\theta,x)\in\Theta\times\mathsf{X}$
and $S\in\mathcal{B}(\Theta)\times\mathsf{\mathcal{B}(X)}$,
\begin{equation}
{\rm \Gamma}\bigl(\theta,x;S\bigr):=\int_{S}\pi_{x}\bigl({\rm d}\vartheta\bigr)\pi_{\vartheta}({\rm d}y)\quad.\label{eq:defGibbs}
\end{equation}
This can be interpreted as a Markov transition probability, and is
precisely the Markov kernel underpinning the standard two stage Gibbs
sampler. The corresponding Markov chain $\{(\theta_{i},X_{i}),i\geq0\}$
on $\Theta\times\mathsf{X}$ leaves $\pi$ invariant and is ergodic
under fairly general and natural conditions. In fact it can be shown
that $\{X_{i},i\geq0\}$ and $\{\theta_{i},i\geq0\}$ are themselves
Markov chains leaving the marginals $\pi\bigl({\rm d}x\bigr)$ and
$\pi({\rm d}\theta)$ invariant respectively. For reasons which will
appear clearer below, we define for any $(x_{0},S)\in\mathsf{X}\times\mathcal{B}\bigl(\mathsf{X}\bigr)$
the Markov transition probability $\Gamma_{x}\bigl(x_{0},S\bigr):=\Gamma\bigl(x_{0},\Theta\times S\bigr)$
corresponding to the Markov chain $\{X_{i},i\geq0\}$ (we point out
that the index $x$ in this notation is a name, not a variable). In
some situations, however, while sampling from the conditional distribution
$\pi_{x}\bigl({\rm d}\theta\bigr)$ may be routine, sampling from
$\pi_{\theta}({\rm d}x)$ may be difficult and this step is instead
replaced by a Markov transition probability $\Pi_{\theta}(x,{\rm d}y)$
leaving $\pi_{\theta}({\rm d}x)$ invariant for any $\theta\in\Theta$.
The resulting algorithm, whose transition kernel $\Phi$ is given
below, is often referred to as ``Metropolis-within-Gibbs'' in the
common situation where $\Pi_{\theta}$ is a Metropolis--Hastings transition
kernel---we will however use this name in order to refer to the general
scenario. In the particular situation where $\Pi_{\theta}$ is a cSMC
transition kernel the resulting algorithm is known as the particle
Gibbs (PGibbs) sampler~\citep{andrieu-doucet-holenstein}. We note
that in the general scenario, for any $(\theta_{0},x,S)\in\Theta\times\mathsf{X}\times\big(\mathcal{B}(\Theta)\times\mathsf{\mathcal{B}(X)}\big)$
\begin{align}
\Phi(x,S)=\Phi(\theta_{0},x;S): & =\int_{S}\pi_{x}\bigl({\rm d}\theta\bigr)\Pi_{\theta}(x,{\rm d}y)\quad.\label{eq:def:met-within-gibbs}
\end{align}
Similarly to above one can show that $\{X_{i},i\geq1\}$ defines a
Markov chain, with transition kernel, for $(x_{0},S)\in\mathsf{X}\times\mathcal{B}\bigl(\mathsf{X}\bigr)$,
$\Phi_{x}(x_{0},S):=\Phi(x_{0},\Theta\times S)$ which is $\pi\bigl({\rm d}x\bigr)-$reversible,
and positive as soon as $\Pi_{\theta}$ defines a positive operator
for any $\theta\in\Theta$. Indeed since for any $f,g\in L^{2}\bigl(\mathsf{X},\pi\bigr)$,
\begin{align*}
\int_{\mathsf{X}}f(x)\pi({\rm d}x)\int_{\Theta\times\mathsf{X}}\pi_{x}\bigl({\rm d}\theta\bigr)\Pi_{\theta}(x,{\rm d}y)g(y) & =\int_{\Theta}\pi({\rm d}\theta)\int_{\mathsf{X}^{2}}f(x)g(y)\pi_{\theta}\bigl({\rm d}x\bigr)\Pi_{\theta}(x,{\rm d}y)\\
 & =\int_{\Theta}\pi({\rm d}\theta)\int_{\mathsf{X}^{2}}f(x)g(y)\pi_{\theta}\bigl({\rm d}y\bigr)\Pi_{\theta}(y,{\rm d}x)\quad,
\end{align*}
we deduce the reversibility from the choice $f(x)=\mathbb{I}\{x\in S_{1}\}$
and $g(x)=\mathbb{I}\{x\in S_{2}\}$ for $S_{1},S_{2}\in\mathcal{B}\bigl(\mathsf{X}\bigr)$
and the positivity by letting $g=f$. This motivates the following
simple result, which again draws on the standard Hilbert space techniques
outlined in~\citep[Appendix~\ref{sec:Supplementary-material-for_minorization}]{andrieu2015uniformsupplement},
and is to the best of our knowledge not available in the literature.
We naturally remark that $\Gamma$ is a particular instance of $\Phi$
corresponding to the case where for any $(\theta,x)\in\Theta\times\mathsf{X}$,
$\Pi_{\theta}(x,\cdot)=\pi_{\theta}(\cdot)$, therefore also implying
that $\Gamma_{x}$ is self-adjoint. Our first result, Theorem~\ref{thm:generalMWGresult},
takes advantage of the fact that $\Gamma_{x}$ is reversible, and
therefore focuses on the asymptotic variance of functions $f\in L^{2}\bigl(\mathsf{X},\pi\big)$.
Corollary~\ref{rem:pgibbs_geometric_result} follows from this result,
providing a sufficient condition for geometric ergodicity of the PGibbs
Markov chain. Our second result, Theorem~\ref{thm:MWG-functions-of-theta},
focuses on functions $g\in L^{2}\bigl(\Theta,\pi\bigr)$, but the
same technique is not directly applicable in this scenario. Some
of our results concern Dirichlet forms: for a generic $\mu$-reversible
Markov kernel and a function $f\in L^{2}(\mathsf{E},\mu)$ we define
the Dirichlet form $\mathcal{E}_{\Pi}(f):=\left\langle f,(I-\Pi)f\right\rangle _{\mu}$.

\begin{thm}
\label{thm:generalMWGresult}Let $\pi$ be a probability distribution
defined on $\bigl(\Theta\times\mathsf{X},\mathcal{B}(\Theta)\times\mathcal{B}(\mathsf{X)}\bigr)$
and let $\left\{ \Pi_{\theta},\theta\in\Theta\right\} $ be a family
of Markov transition probabilities $\left\{ \Pi_{\theta},\theta\in\Theta\right\} $
such that for any $\theta\in\Theta$ the Markov kernel $\Pi_{\theta}$
is reversible with respect to $\pi_{\theta}$, and let $\Gamma$ and
$\Phi$ be as in (\ref{eq:defGibbs}) and (\ref{eq:def:met-within-gibbs}).
Define
\begin{align}
\varrho: & =\inf_{f\in L^{2}\bigl(\mathsf{X},\pi\bigr)}\frac{\int_{\Theta}\pi({\rm d}\theta){\rm var}_{\pi_{\theta}}\left(f\right){\rm Gap}\bigl(\Pi_{\theta}\bigr)}{\int_{\Theta}\pi({\rm d}\theta){\rm var}_{\pi_{\theta}}\left(f\right)}\quad.\label{eq:varrho}
\end{align}
 Then, for any $f\in L^{2}(\mathsf{X},\pi)$ we have the following
inequalities,
\begin{enumerate}[label=(\alph*)]
\item \label{item:dirichlet-order}for the Dirichlet forms,\textup{
\begin{align*}
2\mathcal{E}_{{\rm \Gamma}_{x}}(f)\geq\mathcal{E}_{\Phi_{x}}(f) & \geq\varrho\times\mathcal{E}_{\Gamma_{x}}(f)\quad,
\end{align*}
}
\item for the right spectral gaps\label{item:gap-order}
\[
2{\rm Gap}\left({\rm \Gamma}_{x}\right)\geq{\rm Gap}\bigl(\Phi_{x}\bigr)\geq\varrho\times{\rm Gap}\left(\Gamma_{x}\right)\quad,
\]
 
\item if the asymptotic variances,\label{item:var-order}
\[
0\leq\frac{{\rm var}\bigl(f,\Gamma_{x}\bigr)-{\rm var}_{\pi}(f)}{2}\leq{\rm var}\bigl(f,\Phi_{x}\bigr)\leq(\varrho^{-1}-1){\rm var}_{\pi}(f)+{\rm \varrho^{-1}}{\rm var}\left(f,\Gamma_{x}\right)\quad,
\]
where the latter inequality holds for $\varrho>0.$
\item In addition if\label{item:uniform-pos}

\begin{enumerate}
\item [(i)]\label{item:uniform-order}there exist $\epsilon>0$ such that
for all $\theta\in\Theta$ and all $(x,B)\in\mathsf{X}\times\mathcal{B}\big(\mathsf{X}\big)$,
the minorisation inequality $\Pi_{\theta}\big(x,B\big)\geq\epsilon\pi_{\theta}\big(B\big)$
holds, then for any $f\in L^{2}(\mathsf{X},\pi)$
\[
\frac{{\rm var}\bigl(f,\Gamma_{x}\bigr)-(1-\epsilon){\rm var}_{\pi}(f)}{(2-\epsilon)}\leq{\rm var}\bigl(f,\Phi_{x}\bigr)\;,
\]

\item [(ii)]\label{item:lowerboundvarwithPipositive}for all $\theta\in\Theta$,
$\Pi_{\theta}$ is a positive operator then for any $f\in L^{2}(\mathsf{X},\pi)$
\[
{\rm var}\bigl(f,\Gamma_{x}\bigr)\leq{\rm var}\bigl(f,\Phi_{x}\bigr)\;.
\]

\end{enumerate}
\end{enumerate}
\end{thm}
\begin{proof}
We prove the first point. Without loss of generality we consider any
$f\in L_{0}^{2}\bigl(\mathsf{X},\pi\bigr)$ and notice that 
\[
\mathcal{E}_{\Phi_{x}}(f)=\int_{\Theta}\pi({\rm d}\theta)\mathcal{E}_{\Pi_{\theta}}\bigl(f\bigr)\quad,
\]
since
\begin{align*}
\int_{\Theta\times\mathsf{X}^{2}}\pi({\rm d}x)\pi_{x}({\rm d}\theta)\Pi_{\theta}(x,{\rm d}y)\left[f(x)-f(y)\right]^{2} & =\int_{\Theta}\pi({\rm d}\theta)\int_{\mathsf{X}^{2}}\pi_{\theta}({\rm d}x)\Pi_{\theta}(x,{\rm d}y)\left[f(x)-f(y)\right]^{2}\quad.
\end{align*}
Now using that $\mathcal{E}_{\Gamma_{x}}(f)=\frac{1}{2}\int_{\Theta\times\mathsf{X}^{2}}\pi({\rm d}x)\pi_{x}({\rm d}\theta)\pi_{\theta}({\rm d}y)\left[f(x)-f(y)\right]^{2}=\int_{\Theta}\pi({\rm d}\theta){\rm var}_{\pi_{\theta}}\left(f\right)$
and letting $\bar{f}_{\theta}:=f-\pi_{\theta}(f)$ for any $\theta\in\Theta$,
we obtain
\begin{align*}
\mathcal{E}_{\Phi_{x}}(f) & =\int_{\Theta}\pi({\rm d}\theta){\rm var}_{\pi_{\theta}}\left(f\right)\frac{\int_{\Theta}\pi({\rm d}\theta)\mathcal{E}_{\Pi_{\theta}}\bigl(f\bigr)}{\int_{\Theta}\pi({\rm d}\theta){\rm var}_{\pi_{\theta}}\left(f\right)}\\
 & =\mathcal{E}_{\Gamma_{x}}(f)\times\frac{\int_{\Theta}\pi({\rm d}\theta)\mathbb{I}\{{\rm var}_{\pi_{\theta}}\left(f\right)>0\}{\rm var}_{\pi_{\theta}}\left(f\right)\frac{\mathcal{E}_{\Pi_{\theta}}\bigl(\bar{f}_{\theta}\bigr)}{{\rm var}_{\pi_{\theta}}\left(\bar{f}_{\theta}\right)}}{\int_{\Theta}\pi({\rm d}\theta){\rm var}_{\pi_{\theta}}\left(f\right)}\\
 & \geq\mathcal{E}_{\Gamma_{x}}(f)\times\frac{\int_{\Theta}\pi({\rm d}\theta){\rm var}_{\pi_{\theta}}\left(f\right){\rm Gap}\bigl(\Pi_{\theta}\bigr)}{\int_{\Theta}\pi({\rm d}\theta){\rm var}_{\pi_{\theta}}\left(f\right)}\\
 & \geq\mathcal{E}_{\Gamma_{x}}(f)\times\inf_{g\in L_{0}^{2}\bigl(\mathsf{X},\pi\bigr)}\frac{\int_{\Theta}\pi({\rm d}\theta){\rm var}_{\pi_{\theta}}\left(g\right){\rm Gap}\bigl(\Pi_{\theta}\bigr)}{\int_{\Theta}\pi({\rm d}\theta){\rm var}_{\pi_{\theta}}\left(g\right)}\quad,
\end{align*}
where we have used that for any $g\in L_{0}^{2}\bigl(\mathsf{X},\pi\bigr)$,
$\mathcal{E}_{\Pi_{\theta}}\bigl(g\bigr)\leq2{\rm var}_{\pi_{\theta}}\left(g\right)$
and that the set $A:=\bigl\{\theta\in\Theta:{\rm var}_{\pi_{\theta}}(\bar{f}_{\theta})=\infty\bigr\}$
satisfies $\pi\bigl(A\times\mathsf{X}\bigr)=0$. The latter result
follows from ${\rm var}_{\pi}(f)<\infty$ and the variance decomposition
identity: $\|f\|_{\pi}^{2}=\|f-\bar{f}_{\theta}\|_{\pi}^{2}+\|\bar{f}_{\theta}\|_{\pi}^{2}$.
We deduce \ref{item:dirichlet-order} from the last inequality. Points
\ref{item:gap-order} and \ref{item:var-order} then follow from \citep[Lemma~\ref{lem:minorizationdirichletboundgapvariance}]{andrieu2015uniformsupplement}.

We next turn into \ref{item:uniform-pos}.  As above, we find that
\begin{align*}
\mathcal{E}_{\Phi_{x}}(f) & \leq\mathcal{E}_{\Gamma_{x}}(f)\times\frac{\int_{\Theta}\pi({\rm d}\theta)\mathbb{I}\{{\rm var}_{\pi_{\theta}}\left(f\right)>0\}{\rm var}_{\pi_{\theta}}\left(f\right)\sup_{g\in L_{0}^{2}(\mathsf{X},\pi_{\theta})}\frac{\mathcal{E}_{\Pi_{\theta}}\bigl(g\bigr)}{{\rm var}_{\pi_{\theta}}\left(g\right)}}{\int_{\Theta}\pi({\rm d}\theta){\rm var}_{\pi_{\theta}}\left(f\right)}\quad.
\end{align*}
Under the uniform minorisation condition, we have $\mathcal{E}_{\Pi_{\theta}}\bigl(g\bigr)\le(2-\epsilon)\mathrm{var}_{\pi_{\theta}}(g)$
\citep[Proposition~\ref{prop:boundsspectralgapandvarianceforP_N}]{andrieu2015uniformsupplement},
and consequently $\mathcal{E}_{\Phi_{x}}(f)\leq(2-\varepsilon)\mathcal{E}_{\Gamma_{x}}(f)$.
When $\Pi_{\theta}$ is a positive operator for any $\theta\in\Theta$,
we have $\mathcal{E}_{\Pi_{\theta}}\bigl(g\bigr)\le\mathrm{var}_{\pi_{\theta}}(g)$
and consequently $\mathcal{E}_{\Phi_{x}}(f)\leq\mathcal{E}_{\Gamma_{x}}(f)$.\end{proof}
\begin{rem}
\label{rem:pgibbs_geometric_result}In relation to Theorem \ref{thm:generalMWGresult}
:\end{rem}
\begin{enumerate}[label=(\alph*)]
\item it may be easier in practice to use the lower bound $\underline{\varrho}:=\inf_{\theta\in\Theta}{\rm Gap}\bigl(\Pi_{\theta}\bigr)\leq\varrho$
which leads to ${\rm Gap}\bigl(\Phi_{x}\bigr)\geq\underline{\varrho}\times{\rm Gap}\left(\Gamma_{x}\right)$
and ${\rm var}\bigl(f,\Phi_{x}\bigr)\leq(\underline{\varrho}^{-1}-1){\rm var}_{\pi}(f)+{\rm \underline{\varrho}^{-1}}{\rm var}\left(f,\Gamma_{x}\right)$
when $\underline{\varrho}>0$,
\item one could suggest iterating $\Pi_{\theta}$ sufficiently many times,
say $k_{\theta}$ times, in order to ensure that $\Pi_{\theta}^{k_{\theta}}$
satisfies the uniform in $\theta$ properties of the type suggested
above. This would require however a computable quantitative bound
on the spectral gap of $\Pi_{\theta}$ ,
\item the lower bound in \ref{item:var-order} is motivated by the fact
that $\{\Pi_{\theta},\theta\in\Theta\}$ may be a family with non-positive
elements, which may introduce negative correlations. On the contrary
in the situation where $\{\Pi_{\theta},\theta\in\Theta\}$ is a collection
of positive operators (e.g. cSMC kernels) then \ref{item:gap-order}
implies that $\Phi_{x}$ is geometrically ergodic as soon as $\Gamma_{x}$
is geometrically ergodic and $\varrho>0$ (and of course $\Gamma_{x}$
is always positive) and \ref{item:uniform-pos}(ii) that $\Phi_{x}$
is always inferior to $\Gamma_{x}$ in terms of asymptotic variance.
In the context of the PGibbs sampler the latter result parallels what
is known for pseudo-marginal algorithms~\citep{andrieu2015},
\item we note that from~\citep[Theorem~1; Proposition~1]{roberts2001markov}
$\Phi$ is geometrically ergodic as soon as $\Phi_{x}$ is geometrically
ergodic.
\end{enumerate}

Now we show how these results can be transferred to the $\{\theta_{i}\}$
chain. 
\begin{thm}
\label{thm:MWG-functions-of-theta}Let the notation be as in Theorem~\ref{thm:generalMWGresult}.
Then,
\begin{enumerate}[label=(\alph*)]
\item \label{enu:G-stuff}assume that for some class of functions $\mathcal{G}\subset\{g:\mathsf{X}\to\mathbb{R}\,:\,\pi(|g|)<\infty\}$
there exists a function $\left|\cdot\right|_{\mathcal{G}}:\mathcal{G}\rightarrow[0,\infty]$
and $\rho\in[0,1)$ such that for any probability distribution $\nu$
on $\bigl(\mathsf{X},\mathcal{B}(\mathsf{X})\bigr)$ there exist $W_{\nu}\in[0,\infty]$
such that for all $g\in\mathcal{G}$ and any $k\geq1$  
\[
\left|\nu\Phi_{x}^{k}(g)-\pi(g)\right|\leq\left|g\right|_{\mathcal{G}}W_{\nu}\rho^{k}\quad,
\]
then for any $f:\Theta\rightarrow\mathbb{R}$ such that $\bar{f}(x):=\pi_{x}\bigl(f\bigr)\in\mathcal{G}$
and any $k\geq2$ 
\[
\left|\nu\Phi^{k}(f)-\pi(f)\right|\leq\left|\bar{f}\right|_{\mathcal{G}}W_{\nu}\rho^{k-1}\quad,
\]

\item for any $f\in L^{2}\bigl(\Theta,\pi\bigr)$, letting for any $x\in\mathsf{X}$
$\bar{f}(x):=\pi_{x}\bigl(f\bigr)\in L^{2}\bigl(\mathsf{X},\pi\bigr)$,
we have for any $k\geq1$
\[
\left\langle f,\Phi^{k}f\right\rangle _{\pi}=\left\langle \bar{f},\Phi_{x}^{k-1}\bar{f}\right\rangle _{\pi}
\]
and
\[
{\rm {\rm var}}\bigl(f,\Phi\bigr)={\rm var}_{\pi}\bigl(f\bigr)+{\rm var}_{\pi}\bigl(\bar{f}\bigr)+{\rm var}(\bar{f},\Phi_{x})\quad,
\]

\item if $\varrho>0$ defined in \ref{eq:varrho}, then for $f\in L^{2}\bigl(\Theta,\pi\bigr)$
\begin{align*}
{\rm var}(f,\Phi)\leq & {\rm var}_{\pi}\bigl(f\bigr)+\varrho^{-1}{\rm var}_{\pi}\bigl(\bar{f}\bigr)+{\rm \varrho^{-1}}{\rm var}\left(\bar{f},\Gamma_{x}\right)\\
\leq & (1-\varrho^{-1}){\rm var}_{\pi}\bigl(f\bigr)+{\rm \varrho^{-1}}{\rm var}\bigl(f,\Gamma\bigr)\quad,
\end{align*}

\item if for all $\theta\in\Theta$, $\Pi_{\theta}$ is a positive operator,
then for $f\in L^{2}\bigl(\Theta,\pi\bigr)$ ${\rm var}(f,\Phi)\geq{\rm var}(f,\Gamma)$.
\end{enumerate}
\end{thm}
\begin{proof}
We remark that without loss of generality we can let $f\in L_{0}^{2}(\Theta,\pi)$
throughout. First note that for $f\in L_{0}^{2}\bigl(\Theta,\pi\bigr)$
and any $\bigl(\theta,x_{0}\bigr)\in\Theta\times\mathsf{X}$ 
\begin{align*}
\Phi\bigl(\theta,x_{0};f\bigr) & =\Phi(x_{0};f)=\pi_{x_{0}}\bigl(f\bigr)=\bar{f}(x_{0})\quad,
\end{align*}
and for $g\in L^{2}\bigl(\mathsf{X},\pi\bigr)$ and any $p\geq1$,
$\Phi^{p}(\theta,x_{0};g)=\Phi_{x}^{p}(x_{0};g)$. The first result
is straightforward upon remarking that for $k\geq1$ 
\[
\Phi^{k+1}(x_{0},f)-\pi(f)=\Phi_{x}^{k}(x_{0},\bar{f})-\pi(\bar{f})\quad.
\]
For the second and third point, using the remarks above, for $f\in L_{0}^{2}\bigl(\Theta,\pi\bigr)$
and $k\geq1$ 
\begin{align*}
\left\langle f,\Phi^{k}f\right\rangle _{\pi} & =\left\langle f,\Phi^{k-1}\bar{f}\right\rangle _{\pi}=\left\langle \bar{f},\Phi_{x}^{k-1}\bar{f}\right\rangle _{\pi}\;.
\end{align*}
Now $\|f\|_{\pi}^{2}=\left\langle f-\bar{f}+\bar{f},f-\bar{f}+\bar{f}\right\rangle _{\pi}=\|\bar{f}\|_{\pi}^{2}+\|f-\bar{f}\|_{\pi}^{2}$,
which is the variance decomposition identity and by noting that $\pi\bigl(\bar{f}\bigr)=0$
lets us deduce that $f\in L_{0}^{2}\bigl(\Theta,\pi\bigr)$ implies
that $\bar{f}\in L_{0}^{2}\bigl(\mathsf{X},\pi\bigr)$. Now,
\begin{eqnarray*}
{\rm var}(f,\Phi) & = & \|f\|_{\pi}^{2}+2\sum_{k=1}^{\infty}\left\langle f,\Phi^{k}f\right\rangle _{\pi}=\|f\|_{\pi}^{2}+2\sum_{k=1}^{\infty}\left\langle \bar{f},\Phi_{x}^{k-1}\bar{f}\right\rangle _{\pi}\\
 & = & \|f\|_{\pi}^{2}+2\|\bar{f}\|_{\pi}^{2}+2\sum_{k=1}^{\infty}\left\langle \bar{f},\Phi_{x}^{k}\bar{f}\right\rangle _{\pi}=\|f\|_{\pi}^{2}+\|\bar{f}\|_{\pi}^{2}+{\rm var}\bigl(\bar{f},\Phi_{x}\bigr)\quad.
\end{eqnarray*}
We conclude by noting that for $f\in L^{2}\bigl(\mathsf{X},\pi\bigr)$
then ${\rm var}_{\pi}\bigl(f\bigr)=\|f-\pi(f)\|_{\pi}^{2}$ and ${\rm var}_{\pi}\bigl(\bar{f}\bigr)=\|\bar{f}-\pi(f)\|_{\pi}^{2}=\|\overline{f-\pi(f)}\|_{\pi}^{2}$.
We will also use the equality above for $\Gamma$ and $\Gamma_{x}$,
since again the latter corresponds to a particular instance of the
above. We can now use the bound from Theorem~\ref{thm:generalMWGresult},
which leads, for $f\in L_{0}^{2}\bigl(\Theta,\pi\bigr)$, to
\begin{align*}
{\rm var}(f,\Phi) & \leq\|f\|_{\pi}^{2}+\|\bar{f}\|_{\pi}^{2}+(\varrho^{-1}-1)\|\bar{f}\|_{\pi}^{2}+{\rm \varrho^{-1}}{\rm var}\left(\bar{f},\Gamma_{x}\right)\\
 & =\|f\|_{\pi}^{2}+\varrho^{-1}\|\bar{f}\|_{\pi}^{2}+{\rm \varrho^{-1}}{\rm var}\left(\bar{f},\Gamma_{x}\right)\quad.
\end{align*}
From the remark above we deduce that 
\begin{align*}
\|f\|_{\pi}^{2}+\varrho^{-1}\|\bar{f}\|_{\pi}^{2}+{\rm \varrho^{-1}}{\rm var}\left(\bar{f},\Gamma_{x}\right) & \leq\|f\|_{\pi}^{2}+\varrho^{-1}\|\bar{f}\|_{\pi}^{2}+{\rm \varrho^{-1}}\big[{\rm var}(f,\Gamma)-\|f\|_{\pi}^{2}-\|\bar{f}\|_{\pi}^{2}\big]\\
 & =(1-\varrho^{-1})\|f\|_{\pi}^{2}+{\rm \varrho^{-1}}{\rm var}(f,\Gamma)\quad.
\end{align*}
We conclude as above. The final statement follows from ${\rm var}\bigl(\bar{f},\Phi_{x}\bigr)\geq{\rm var}\bigl(\bar{f},\Gamma_{x}\bigr)$
(see Theorem~\ref{thm:generalMWGresult}) and the equality established
above for $\Phi$ and $\Phi_{x}$ and $\Gamma$ and $\Gamma_{x}$.\end{proof}
\begin{cor}
Consider the PGibbs sampler with $N\geq2$ particles with kernel $\Phi_{N}$
defined as in (\ref{eq:def:met-within-gibbs}) such that for any $\theta\in\Theta$,
$\Pi_{\theta}=P_{\theta,N}$ is the i-cSMC kernel as defined in Section~\ref{sec:The-i-CSMC}
for the families $\{M_{\theta,t}\}$and $\{G_{\theta,t}\}$ of kernels
and potentials on $\mathsf{Z}\times\mathcal{B}\bigl(\mathsf{Z}\bigr)$
and $\mathsf{Z}$ respectively. For any $\theta\in\Theta$ we let
$\gamma_{\theta,T}$ be the corresponding normalizing constant as
defined below (\ref{eq:defofpiforSMCframework}). Then, the results
of Theorems \ref{thm:generalMWGresult} and \ref{thm:MWG-functions-of-theta}
hold as follows:
\begin{enumerate}[label=(\alph*)]
\item if 
\[
\pi-{\rm ess}\sup_{\theta}\frac{\prod_{t=1}^{T}\bar{G}_{\theta,t}}{\gamma_{\theta,T}}<\infty\quad,
\]
then $\varrho\ge\epsilon_{N}$ as defined in Corollary~\ref{cor:convergence_epsilon},
\item or we have the uniform mixing condition, for some $0\leq\alpha<\infty$,
\[
\pi-{\rm ess}\sup_{\theta,z}\frac{Q_{\theta,p,p+k}(1)(z)}{\eta_{\theta,p}Q_{\theta,p,p+k}(1)}\leq\alpha\quad,
\]
then $\varrho\ge\epsilon_{N}$ as defined in Corollary~\ref{cor:linear_in_T_epsilon_N_bound}.
\end{enumerate}

 \noindent In particular, in both cases $\varrho$ convergences
to one as $N\rightarrow\infty$, implying that the spectral gaps and
the asymptotic variances associated with the PGibbs sampler converge
to those of the related Gibbs sampler.

\end{cor}
\begin{rem}
It is worth noting that terms related to $\gamma_{\theta,T}$ appear
in all these bounds. So, for example in the first part it is not sufficient
that our potentials $\{G_{\theta,t}\}$ are essentially bounded, but
it is sufficient if, for all $t\in[T]$, $\pi_{t}-{\rm ess}\sup_{\theta,x_{t}}G_{\theta,t}(x_{t})/\eta_{\theta,t}(G_{t})$
is bounded.

\end{rem}

\section{Discussion\label{sec:Discussion}}

The developments above go some way in characterizing the behaviour
of i-cSMC and associated PGibbs Markov chains, and raise a number
of possible future directions for research. We have already embarked
upon investigating some potentially practical uses of the minorization
conditions and spectral properties for these chains. Of particular
interest in practice is how to choose $N$ in the i-cSMC algorithm
so as to balance the trade off between mixing properties of $P_{N}$
and the total number of iterations that can be performed with limited
computational resources. Remark~\ref{rem:tuning}, for example, can
be used to find approximately good values of $N$ in this spirit,
but can only serve as a heuristic. In particular, while Proposition~\ref{prop:uniformminorizationcrude}
may provide a fairly accurate bound in the large $N$ regime, it is
unclear how much is lost in applying Jensen's inequality, and consequently
how accurate estimates such as those in Remark~\ref{rem:tuning}
can be. It is possible that results such as those in~\citep{berard2013lognormal}
may provide a way to exploit additional structure often found in statistical
applications.

The results for the i-cSMC and PGibbs Markov chains developed here
can be compared and contrasted with similar results for the Particle
Independent Metropolis--Hastings (PIMH) and PMMH Markov chains~\citep{andrieu-doucet-holenstein}.
We summarize here the detailed comparison provided in~\citep[Appendix~\ref{sec:Comparison-with-Particle}]{andrieu2015uniformsupplement}.
Like i-cSMC, PIMH is an exact approximation of an independent sampler
but PMMH is an exact approximation of an idealized Metropolis--Hastings
kernel, rather than a Gibbs sampler. Just as i-cSMC can be viewed
as a constituent element of PGibbs, PIMH can be viewed as playing
the same role within PMMH. Central to the analysis of PIMH is the
essential supremum of the normalizing constant estimate $\hat{\gamma}_{T}^{N}\bigl(Z_{1:T}\bigr)$
introduced in Section~\ref{sec:Minorization-and-Dirichlet} with
respect to the law of a standard SMC algorithm and indeed the PIMH
Markov chain is (uniformly) geometrically ergodic if and only if this
supremum is finite as a consequence of the characterisation of independent
Metropolis--Hastings chains in~\citep{mengersen-tweedie}. However,
it can also be seen that the rate of convergence of PIMH will typically
not improve as $N$ increases, in contrast with the convergence for
the i-cSMC (see Propositions~\ref{prop:boundepsilonNwithuniformboundG}
and~\ref{prop:mixing_bound}).

For PMMH,~\citep{andrieu2015} show that if the essential supremum
of the relative normalizing constant estimate $\hat{\gamma}_{\theta,T}^{N}\bigl(Z_{1:T}\bigr)/\gamma_{\theta,T}$
is moreover bounded essentially uniformly in $\theta$ then the existence
of a spectral gap of the idealized Metropolis--Hastings Markov kernel
it approximates is inherited by PMMH. However, the rate of convergence
of the PMMH Markov chain when this occurs does not improve in general
as $N$ increases, in contrast to our results for PGibbs Markov chains.
In this context, weak convergence in $N$ of the asymptotic variance
of estimates of $\pi(f)$ to the corresponding asymptotic variance
of the Metropolis--Hasting kernel is nevertheless provided by~\citep[Proposition~19]{andrieu2015}
for all $f\in L^{2}(\Theta,\pi)$ but this can be contrasted with
quantitative bounds obtained in Theorem~\ref{thm:MWG-functions-of-theta}.

The one step uniform minorization condition in Corollary~\ref{cor:uniformminorbymu},
where the minorization measure is the invariant distribution of the
Markov chain, suggests that it may be possible to apply coupling from
the past techniques (see, e.g.,~\citep{propp1996exact,murdoch:green:1998,hobert2004mixture})
in order to produce samples from exactly this distribution. It is,
however, not clear how to implement such an algorithm in general,
 although~\citep{lee2014perfect} provides a perfect simulation
algorithm motivated by Theorem~\ref{thm:THEtheorem}. Finally, our
analysis has focused mainly on the case where the essential boundedness
condition holds. However, a refined analysis may permit characterization
of the i-cSMC and hence the PGibbs Markov chains even in the absence
of this condition, with parallels to~\citep{andrieu2015}.
\begin{acknowledgement*}
CA's research was supported by EPSRC EP/K009575/1 Bayesian Inference
for Big Data with Stochastic Gradient Markov Chain Monte Carlo and
EP/K0\-14463/1 Intractable Likelihood: New Challenges from Modern
Applications (ILike). MV was supported by Academy of Finland grant
250575.
\end{acknowledgement*}

\pagebreak{}

\part*{Supplementary material}

\appendix

\section{Proof of Lemma~\ref{lem:simple_correspondence} \label{sec:Proof-of-Lemma_simplecorrespondence}}

The proof of Lemma~\ref{lem:simple_correspondence} is a simple
consequence of Lemma~\ref{lem:correspondenceP_1andP_2} \ref{enu:E-I}.
We introduce the set of indices $\mathcal{J}_{T}:=\bigcup_{m=0}^{T}\{1\}^{m}\times\{2,\ldots,N\}^{T-m}$,
which will allow us to define the lineages coalescing with $\mathbf{1}\in\{1\}^{T}$
at some point in the past, and $m_{\mathbf{i}}:=\max\{k:i_{k}=1\}$
(with the convention that $\max\emptyset=0$) the time at which coalescence
occurs. 
\begin{lem}
\label{lem:correspondenceP_1andP_2}For any $x\in\mathsf{X}$, $z_{1:T}\in\mathsf{X}^{T}$
and $a_{1:T}\in[N]^{N(T-1)}\times[N]$,
\begin{enumerate}[label=(\alph*)]
\item for any\textup{ $y_{2:T}\in\mathsf{Z}^{T-1}$ and }$\mathbf{k}=k_{1:T}\in[N]^{T}$
such that $k^{1}\neq1$ 
\[
\mathbb{P}_{\mathbf{1},x}^{N}\left(Z_{1}\in{\rm d}z_{1}\right)=\int_{\mathsf{Z}}M_{1}\bigl({\rm d}y_{1}\bigr)\mathbb{P}_{\mathbf{1},x,\mathbf{k},y}^{N}\left(Z_{1}\in{\rm d}z_{1}\right)\quad,
\]
and for $t\in\{2,\ldots,T\}$, any $(y_{1},\ldots,y_{t-1},y_{t+1},\ldots,y_{T})\in\mathsf{Z}^{T-1}$,
$\mathbf{k}\in[N]^{T}$ such that $k_{t}\neq1$ and $a_{t-1}^{k_{t}}=k_{t-1}$
\begin{multline*}
\mathbb{P}_{\mathbf{1},x}^{N}\left(Z_{t}\in{\rm d}z_{t},A_{t-1}=a_{t-1}\left|Z_{t-1}=z_{t-1}\right.\right)=\int_{\mathsf{Z}}\frac{G_{t-1}\bigl(z_{t-1}^{k_{t-1}}\bigr)}{\sum_{j=1}^{N}G_{t-1}\bigl(z_{t-1}^{j}\bigr)}M_{t}\bigl(z_{t-1}^{k_{t-1}},{\rm d}y_{t}\bigr)\\
\times\mathbb{P}_{\mathbf{1},x,\mathbf{k},y}^{N}\left(Z_{t}\in{\rm d}z_{t},A_{t-1}=a_{t-1}\left|Z_{t-1}=z_{t-1}\right.\right)\quad.
\end{multline*}

\item \label{enu:E-I}for $\mathbf{i}\in\mathcal{J}_{T}$ and $y_{1:m_{\mathbf{i}}}=x_{1:m_{\mathbf{i}}}$
we have \textup{
\begin{multline*}
\mathbb{E}_{\mathbf{1},x}^{N}\left[I_{\mathbf{i}}\bigl(Z_{1:T},A_{1:T},S\bigr)\right]\\
=\int_{\mathsf{Z}^{T-m_{\mathbf{i}}}}M_{m_{\mathbf{i}},T}(x_{m_{\mathbf{i}}},{\rm d}y_{m_{\mathbf{i}}+1:T})\times\mathbb{E}_{\mathbf{1},x,\mathbf{i},y}^{N}\left[\frac{{\textstyle {\displaystyle {\textstyle \prod_{t=m_{\mathbf{i}}}^{T}}G_{t}(y_{t})}}\times\mathbb{I}\{y\in S\}}{{\displaystyle {\textstyle \prod_{t=m_{\mathbf{i}}}^{T}\sum_{j=1}^{N}}G_{t}(Z_{t}^{j})}}\right].
\end{multline*}
}
\item for $\mathbf{i}\notin\mathcal{J}_{T}$\textup{, $\mathbb{E}_{\mathbf{1},x}^{N}\left[I_{\mathbf{i}}\bigl(Z_{1:T},A_{1:T},S\bigr)\right]=0$.}
\end{enumerate}
\end{lem}
We note that the above is well defined for $m_{\mathbf{i}}=0$ from
the definition of $M_{p,l}$ in Section~\ref{sec:The-i-CSMC} and
associated remark, and the convention that $x_{1:0}=y_{1:0}$ should
be ignored in this case.
\begin{proof}[Proof of Lemma~\ref{lem:correspondenceP_1andP_2}]
In order to alleviate notation we omit $Z_{t}\in\cdot,Z_{t-1}=\cdot$
and $A_{t-1}=\cdot$ and set $G_{t}^{k}:=G_{t}\bigl(z_{t}^{k}\bigr)$.
For the first point we note the independence on $(y_{2},\ldots,y_{T})\in\mathsf{Z}^{T-1}$
of 
\[
\mathbb{P}_{\mathbf{1},x,\mathbf{k},y}^{N}\left({\rm d}z_{1}\right)=\delta_{x_{1},y_{1}}\bigl({\rm d}z_{1}^{1}\times{\rm d}z_{1}^{k_{1}}\bigr)\prod_{i=2,i\neq k_{1}}^{N}M_{1}({\rm d}z_{1}^{i})\quad,
\]
and then since $k_{1}\neq1$,
\[
\int_{\mathsf{Z}}M_{1}\bigl({\rm d}y_{1}\bigr)\delta_{x_{1},y_{1}}\bigl({\rm d}z_{1}^{1}\times{\rm d}z_{1}^{k_{1}}\bigr)\prod_{i=2,i\neq k_{1}}^{N}M_{1}({\rm d}z_{1}^{i})=\delta_{x_{1}}\bigl({\rm d}z_{1}^{1}\bigr)\prod_{i=2}^{N}M_{1}({\rm d}z_{1}^{i})
\]
and conclude from (\ref{eq:def_P_=00007B1,x=00007D-time-1}). Similarly
we note the independence on $(y_{1},\ldots,y_{t-1},y_{t+1},\ldots,y_{T})\in\mathsf{Z}^{T-1}$
of
\[
\delta_{x_{t},y_{t}}\bigl({\rm d}z_{t}^{1}\times{\rm d}z_{t}^{k_{t}}\bigr)\times\mathbb{I}\{a_{t-1}^{1,k_{t}}=(1,k_{t-1})\}\prod_{i=2,i\neq k_{t}}^{N}\frac{G_{t-1}^{a_{t-1}^{i}}}{\sum_{j=1}^{N}G_{t-1}^{j}}M_{t}(z_{t-1}^{a_{t-1}^{i}},{\rm d}z_{t}^{i})
\]
 (we note however that we will have $\mathbb{P}_{\mathbf{1},x,\mathbf{k},y}\bigl(Z_{t-1}^{k_{t-1}}\in{\rm d}y_{t-1}\bigr)=1$)
and since $k_{t}\neq1$
\begin{multline*}
\mathbb{I}\{a_{t-1}^{1,k_{t}}=(1,k_{t-1})\}\int_{\mathsf{Z}}\frac{G_{t-1}^{k_{t-1}}}{\sum_{k=1}^{N}G_{t-1}^{k}}M_{t}\bigl(z_{t-1}^{k_{t-1}},{\rm d}y_{t}\bigr)\delta_{x_{t},y_{t}}\bigl({\rm d}z_{t}^{1}\times{\rm d}z_{t}^{k_{t}}\bigr)\prod_{i=2,i\neq k_{t}}^{N}\frac{G_{t-1}^{a_{t-1}^{i}}}{\sum_{j=1}^{N}G_{t-1}^{j}}M_{t}(z_{t-1}^{a_{t-1}^{i}},{\rm d}z_{t}^{i})\\
=\mathbb{I}\{a_{t-1}^{1}=1\}\delta_{x_{t}}\bigl({\rm d}z_{t}^{1}\bigr)\prod_{i=2}^{N}\frac{G_{t-1}^{a_{t-1}^{i}}}{\sum_{j=1}^{N}G_{t-1}^{j}}M_{t}(z_{t-1}^{a_{t-1}^{i}},{\rm d}z_{t}^{i})
\end{multline*}
and we conclude with (\ref{eq:def_P_=00007B1,x=00007D-other-times}).
For the second point, let $\mathbf{i}\in\mathcal{J}_{T}$, $a_{1:T}\in\bigl([N]^{N}\bigr)^{T-1}\times[N]$
such that $a_{t-1}^{i_{t}}=i_{t-1}$ for $t=m_{\mathbf{i}}+1,\ldots,T$,
$a_{T}=i_{T}$ and $y_{1:m_{\mathbf{i}}}=x_{1:m_{\mathbf{i}}}$ then,
with an obvious convention when $m_{\mathbf{i}}=1$ (i.e. $a_{0}$
does not exist and should be ignored), we have
\begin{align*}
{\textstyle \int_{\mathsf{Z}^{T-m_{\mathbf{i}}}}} & \mathbb{I}\{y\in S\}\times{\textstyle {\displaystyle \prod_{t=m_{\mathbf{i}}+1}^{T}}}M_{t}(y_{t-1},{\rm d}y_{t})\frac{G_{t-1}(y_{t-1})}{{\textstyle \sum_{j=1}^{N}}G_{t-1}^{j}}\times\mathbb{P}_{\mathbf{1},x,\mathbf{i},y}^{N}\left({\rm d}z_{m_{\mathbf{i}}+1:T},a_{m_{\mathbf{i}}:T-1}\left|z_{m_{\mathbf{i}}}\right.\right)\\
= & \int_{\mathsf{Z}^{T-m_{\mathbf{i}}}}{\displaystyle \prod_{t=m_{\mathbf{i}}+1}^{T}M_{t}(y_{t-1},{\rm d}y_{t})\frac{G_{t-1}^{i_{t-1}}}{{\textstyle \sum_{j=1}^{N}}G_{t-1}^{j}}}\mathbb{I}\{y\in S,a_{t-1}^{i_{t}}=i_{t-1}\}\mathbb{P}_{\mathbf{1},x,\mathbf{i},y}^{N}\left({\rm d}z_{t},a_{t-1}\left|z_{t-1}\right.\right)\\
= & \mathbb{I}\{(y_{1:m_{\mathbf{i}},}z_{m_{\mathbf{i}}+1:T}^{i_{m_{\mathbf{i}}+1:T}})\in S\}{\textstyle {\displaystyle \prod_{t=m_{\mathbf{i}}+1}^{T}}\mathbb{P}_{\mathbf{1},x}^{N}\left({\rm d}z_{t},a_{t-1}\left|z_{t-1}\right.\right)\mathbb{I}\{a_{t-1}^{i_{t}}=i_{t-1}\}}\\
= & \mathbb{I}\{(y_{1:m_{\mathbf{i}},}z_{m_{\mathbf{i}}+1:T}^{i_{m_{\mathbf{i}}+1:T}})\in S\}{\textstyle {\displaystyle \prod_{t=m_{\mathbf{i}}+1}^{T}}\mathbb{I}\{a_{t-1}^{i_{t}}=i_{t-1}\}}\mathbb{P}_{\mathbf{1},x}^{N}\left({\rm d}z_{m_{\mathbf{i}}+1:T},a_{m_{\mathbf{i}}:T-1}\left|z_{m_{\mathbf{i}}}\right.\right),
\end{align*}
where we have used the fact that from the structure of $\mathbb{P}_{\mathbf{1},x,\mathbf{i},y}^{N}\bigl(\cdot\bigr)$
we have $z_{m_{\mathbf{i}}+1:T}^{i_{m_{\mathbf{i}}+1:T}}=y_{m_{\mathbf{i}+1}:T}$.
We notice that $\mathbb{P}_{\mathbf{1},x,\mathbf{k},y}^{N}\left(A_{T}=k\left|Z_{T}=z_{T}\right.\right)=\mathbb{P}_{\mathbf{1},x}^{N}\left(A_{T}=k\left|Z_{T}=z_{T}\right.\right)$
and conclude from the definition of $\mathbb{P}_{\mathbf{1},x}^{N}\bigl(\cdot\bigr)$.
For the third point we remark that for any $z_{1:T},a_{1:T},S\in\bigl(\mathsf{Z}^{N}\bigr)^{T}\times\bigl([N]^{N}\bigr)^{T-1}\times[N]\times\mathcal{B}\bigl(\mathsf{X}\bigr)$
such that $a_{1:T}^{{\bf 1}}\in\{1\}^{T}$ then $I_{{\bf i}}(z_{1:T},a_{1:T},S)=0$
if $\mathbf{i}\notin\mathcal{J}_{T}$ and the result follows from
the definition of $\mathbb{E}_{\mathbf{1},x,\mathbf{i},y}^{N}\bigl(\cdot\bigr)$.
\end{proof}

\section{Proof of Lemma~\ref{lem:selfajoint_positive} \label{sec:Proof-of-Lemma_selfadjoint}}
\begin{proof}[Proof of Lemma~\ref{lem:selfajoint_positive}]
We can define the artificial joint distribution
\begin{equation}
\tilde{\pi}({\bf k},{\rm d}z_{1:T},a_{1:T-1}):=\frac{1}{N^{T}}\int_{\mathsf{X}}\pi({\rm d}x)\mathbb{P}_{{\bf k},x}\left(Z\in{\rm d}z_{1:T},A_{1:T-1}=a_{1:T-1}\right).\label{eq:definitionartificialdistribution}
\end{equation}
This admits as a marginal 
\[
\tilde{\pi}({\rm d}z_{1:T},a_{1:T-1})=\sum_{{\bf k}\in[N]^{T}}\frac{1}{N^{T}}\int_{\mathsf{X}}\mathbb{P}_{{\bf k},x}\left(Z\in{\rm d}z_{1:T},A_{1:T-1}=a_{1:T-1}\right)\pi({\rm d}x).
\]
It is straightforward to check that the conditional distribution of
${\bf K}$ given $(z_{1:T},a_{1:T-1})$ can be written
\[
\tilde{\pi}_{z_{1:T},a_{1:T-1}}({\bf k})=\frac{G_{T}(z_{T}^{k_{T}})}{\sum_{j=1}^{N}G_{T}(z_{T}^{j})}\prod_{t=2}^{T}\mathbb{I}\left\{ k_{t-1}=a_{t-1}^{k_{t}}\right\} .
\]
Indeed, we can define the Markov kernel $\tilde{P}_{N}$ 
\[
\tilde{P}_{N}(x,S):=\sum_{{\bf k}\in[N]^{T}}\frac{1}{N^{T}}\sum_{{\bf i}\in[N]^{T}}\int_{\mathsf{X}^{T}\times[N]^{T-1}}\mathbb{P}_{{\bf k},x}\left(Z_{1:T}\in{\rm d}z_{1:T},A_{1:T-1}=a_{1:T-1}\right)\pi_{z_{1:T},a_{1:T-1}}({\bf i})\mathbb{I}\left\{ z_{1:T}^{{\bf i}}\in S\right\} .
\]
The interpretation of this kernel is that it simulates from the conditional
distribution of $\left(Z_{1:T}^{-{\bf k}},A_{1:T-1}\right)$ given
$({\bf k},Z_{1:T}^{{\bf k}})$ and then draws ${\bf K}={\bf i}$ conditional
upon $\left(Z_{1:T},A_{1:T-1}\right)$, returning $Z_{1:T}^{{\bf i}}$.
This provides immediately that $\tilde{P}_{N}$ is a self-adjoint,
positive operator on $L^{2}(\mathsf{X},\pi)$ (see Appendix \ref{sec:Supplementary-material-for_minorization})
since
\begin{eqnarray*}
 &  & \left\langle \tilde{P}_{N}f,g\right\rangle _{\pi}\\
 & = & \int_{\mathsf{X}^{2}}g(x)f(y)\pi({\rm d}x)\tilde{P}_{N}(x,{\rm d}y)\\
 & = & \int_{\mathsf{X}^{2}}g(x)f(y)\pi({\rm d}x)\\
 &  & \sum_{{\bf k}\in[N]^{T}}\frac{1}{N^{T}}\sum_{{\bf i}\in[N]^{T}}\int_{\mathsf{X}^{T}\times[N]^{T-1}}\mathbb{P}_{{\bf k},x}\left(Z\in{\rm d}z_{1:T},A_{1:T-1}=a_{1:T-1}\right)\pi_{z_{1:T},a_{1:T-1}}({\bf i})\delta_{z_{1:T}^{{\bf i}}}({\rm d}y)\\
 & = & \sum_{{\bf k}\in[N]^{T}}\sum_{{\bf i}\in[N]^{T}}\int_{\mathsf{X}^{T}\times[N]^{T-1}}g(z_{1:T}^{{\bf k}})f(z_{1:T}^{{\bf i}})\tilde{\pi}({\bf k},{\rm d}z_{1:T},a_{1:T-1})\pi_{z_{1:T},a_{1:T-1}}({\bf i})\\
 & = & \sum_{{\bf k}\in[N]^{T}}\sum_{{\bf i}\in[N]^{T}}\int_{\mathsf{X}^{T}\times[N]^{T-1}}g(z_{1:T}^{{\bf k}})f(z_{1:T}^{{\bf i}})\tilde{\pi}({\rm d}z_{1:T},a_{1:T-1})\pi_{z_{1:T},a_{1:T-1}}({\bf k})\pi_{z_{1:T},a_{1:T-1}}({\bf i}).
\end{eqnarray*}
Self-adjointness of $\tilde{P}_{N}$ follows, since clearly $\left\langle \tilde{P}_{N}f,g\right\rangle _{\pi}=\left\langle f,\tilde{P}_{N}g\right\rangle _{\pi}$
and the positivity follows because
\[
\left\langle \tilde{P}_{N}f,f\right\rangle _{\pi}=\int_{\mathsf{X}^{T}\times[N]^{T-1}}\tilde{\pi}({\rm d}z_{1:T},a_{1:T-1})\pi_{z_{1:T},a_{1:T-1}}\left(\tilde{f}_{z_{1:T}}\right)^{2}\geq0,
\]
where $\tilde{f}_{z_{1:T}}({\bf k}):=f(z_{1:T}^{{\bf k}})$. \\
\\
In fact, when we implement the algorithm, we do not use $\tilde{P}_{N}$.
However, we have
\begin{eqnarray*}
P_{N}(x,S) & = & \mathbb{E}_{{\bf 1},x}\left[\sum_{{\bf i}\in[N]^{T}}I_{{\bf i}}(Z_{1:T},A_{1:T},S)\right]\\
 & = & \mathbb{E}_{{\bf k},x}\left[\sum_{{\bf i}\in[N]^{T}}I_{{\bf i}}(Z_{1:T},A_{1:T},S)\right],
\end{eqnarray*}
for any ${\bf k}\in[N]^{T}$ in the case of multinomial resampling.
 (see, e.g.,~\citep{chopin:singh:2013}), and as a consequence,
$P_{N}(x,S)=\tilde{P}_{N}(x,S)$. 
\end{proof}

\section{Supplementary material for Section~\ref{sec:Minorization-and-Dirichlet}
\label{sec:Supplementary-material-for_minorization}}

In the next proposition we gather general properties for generic reversible
Markov chains satisfying a uniform minorization condition for which
the minorization probability is precisely the invariant distribution
of the Markov chain. We suspect these results to be widely known,
but could not find a relevant reference. Let $L^{2}(\mathsf{E},\mu)$
and $L_{0}^{2}\bigl(\mathsf{E},\mu\bigr):=\bigl\{ f\in L^{2}(\mathsf{E},\mu):\mu\bigl(f\bigr)=0\bigr\}$
both endowed with the inner product defined for any $f,g\in L^{2}(\mathsf{E},\mu)$
as $\left\langle f,g\right\rangle _{\mu}:=\int_{\mathsf{E}}f(x)g(x)\mu({\rm d}x)$,
which yields the associated norm $\|f\|_{\mu}:=\sqrt{\left\langle f,f\right\rangle _{\mu}}$.
For any $f\in L^{2}(\mathsf{E},\mu)$ we define the Dirichlet forms

\begin{align*}
\mathcal{E}_{\Pi}(f) & :=\left\langle f,(I-\Pi)f\right\rangle _{\mu}\quad,
\end{align*}
where $I$ is the identity operator. The right and left spectral gaps
of a generic reversible Markov transition kernel have the following
variational representation 
\[
{\rm Gap}\left(\Pi\right):=\inf_{f\in L_{0}^{2}(\mathsf{E},\mu)}\frac{\mathcal{E}_{\Pi}(f)}{\|f\|_{\mu}^{2}}\;\text{and}\;{\rm Gap}_{L}\left(\Pi\right):=2-\sup_{f\in L_{0}^{2}(\mathsf{E},\mu)}\frac{\mathcal{E}_{\Pi}(f)}{\|f\|_{\mu}^{2}}\quad.
\]
The condition ${\rm Gap}\left(\Pi\right)>0$ and ${\rm Gap}_{L}\left(\Pi\right)>0$
implies geometric ergodicity of the Markov chain. It turns out that
convergence is in fact uniformly geometric in the following scenario.
\begin{prop}
\label{prop:boundsspectralgapandvarianceforP_N}Let $\mu$ be a probability
distribution on some measurable space $\bigl(\mathsf{E},\mathcal{B}\bigl(\mathsf{E}\bigr)\bigr)$
and let $\Pi:\mathsf{E}\times\mathcal{B}\bigl(\mathsf{E}\bigr)\rightarrow[0,1]$
be a Markov transition kernel reversible with respect to $\mu$.
Assume that there exists $\varepsilon>0$ such that for any $(x,A)\in\mathsf{E}\times\mathcal{B}\bigl(\mathsf{E}\bigr)$,
\[
\Pi(x,A)\geq\varepsilon\mu(A)\quad,
\]
then
\begin{enumerate}[label=(\alph*)]
\item \label{enu:Dirichlet-var}the Dirichlet forms satisfy for any $f\in L^{2}\bigl(\mathsf{E},\mu\bigr)$
\begin{align*}
\varepsilon\mathrm{var}_{\mu}(f)\leq\mathcal{E}_{\Pi}(f) & \le(2-\varepsilon)\mathrm{var_{\mu}(f)}\;,
\end{align*}

\item \label{enu:the-spectral-gaps}the spectral gaps are lower bounded
by
\[
\min\left\{ {\rm Gap}\bigl(\Pi\bigr),{\rm Gap}_{L}\left(\Pi\right)\right\} \geq\varepsilon,
\]

\item \label{enu:rudolfclassicalresult}for any probability distribution
$\nu\ll\mu$  and any $k\in\mathbb{N}$, 
\[
\|\nu\Pi^{k}\bigl(\cdot\bigr)-\mu\bigl(\cdot\bigr)\|_{L^{2}(\mathsf{E},\mu)}\leq\|\nu-\mu\|_{L^{2}(\mathsf{E},\mu)}(1-\varepsilon)^{k},
\]
 
\item \label{enu:TVvariationconvergenceKontoMeyn}for any probability distribution
$\nu\ll\mu$   we have
\[
\|\nu\Pi^{k}\bigl(\cdot\bigr)-\mu\bigl(\cdot\bigr)\|_{TV}\leq\frac{1}{2}\|\nu-\mu\|_{L^{2}(\mathsf{E},\mu)}\left(1-\varepsilon\right)^{k}\quad,
\]

\item \label{enu:uniformdoeblin}for any $x\in\mathsf{X}$,
\[
\|\delta_{x}\Pi^{k}\bigl(\cdot\bigr)-\mu\bigl(\cdot\bigr)\|_{TV}\leq\left(1-\varepsilon\right)^{k}\quad,
\]

\item and for any $f\in L^{2}\bigl(\mathsf{E},\mu\bigr)$
\[
\frac{\varepsilon}{2-\varepsilon}{\rm var}_{\mu}\bigl(f\bigr)\leq{\rm var}\bigl(f,\Pi\bigr)\leq\left(2\varepsilon^{-1}-1\right){\rm var}_{\mu}\bigl(f\bigr)\quad.
\]
and if $\Pi$ is a positive operator then naturally ${\rm var}\bigl(f,\Pi\bigr)\geq{\rm var}_{\mu}\bigl(f\bigr)$.
\end{enumerate}
\end{prop}
\begin{proof}[Proof of Proposition~\ref{prop:boundsspectralgapandvarianceforP_N}]
First, from the minorization condition one can write $\Pi(x,{\rm d}y)=\varepsilon\mu({\rm d}y)+(1-\varepsilon)R_{\Pi,\varepsilon}(x,{\rm d}y)$,
where $R_{\Pi,\varepsilon}(x,A):=\frac{\Pi(x,A)-\varepsilon\mu(A)}{1-\varepsilon}$
is $\mu-$invariant. Now for $f\in L_{0}^{2}\big(\mathsf{E},\mu\big)$
\begin{eqnarray*}
\left\langle f,\Pi f\right\rangle _{\mu} & = & \varepsilon\left\langle f,\mu(f)\right\rangle _{\mu}+(1-\varepsilon)\left\langle f,R_{\Pi,\varepsilon}f\right\rangle _{\mu}\\
 & = & (1-\varepsilon)\left\langle f,R_{\Pi,\varepsilon}f\right\rangle _{\mu}
\end{eqnarray*}
and therefore with $\mathcal{E}_{\mu}(f)=\left\langle f,f\right\rangle _{\mu}$
the Dirichlet form of the (reversible) ``independent samples'' Markov
chain we deduce 
\begin{align*}
\varepsilon\mathcal{E}_{\mu}(f)\leq\mathcal{E}_{\Pi}(f) & \leq(2-\varepsilon)\left\langle f,f\right\rangle _{\mu}=(2-\varepsilon)\mathcal{E}_{\mu}(f)\;,
\end{align*}
which implies \ref{enu:Dirichlet-var}. The bounds on the spectral
gaps \ref{enu:the-spectral-gaps} follow immediately and the results
in points~\ref{enu:rudolfclassicalresult} and~\ref{enu:TVvariationconvergenceKontoMeyn}
are now a consequence of the resulting property of the spectrum and
e.g.~\citep[Proposition 3.12, p. 44]{rudolf2011explicit} and~\citep[Proposition 1.5]{kontoyiannis:meyn:2012}.
Result~\ref{enu:uniformdoeblin} is due to Doeblin~\citep{lindvall},
while the two bounds on the asymptotic variance are direct consequences
of Lemma~\ref{lem:minorizationdirichletboundgapvariance} and coincide
in this case with the ``Kipnis--Varadhan'' upper bound \citep{kipnis1986central}.
\end{proof}
\begin{lem}
\label{lem:minorizationdirichletboundgapvariance}Let $\Pi_{1},\Pi_{2}$
be reversible with respect to $\mu$ and assume that there exists
$\varrho\ge0$ such that for any $f\in L_{0}^{2}\bigl(\mathsf{E},\mu\bigr)$
\[
\mathcal{E}_{\Pi_{2}}\bigl(f\bigr)\geq\varrho\mathcal{E}_{\Pi_{1}}\bigl(f\bigr)\quad,
\]
then
\[
{\rm Gap}\bigl(\Pi_{2}\bigr)\geq\varrho{\rm Gap}\bigl(\Pi_{1}\bigr)\quad,
\]
and if $\varrho>0$
\[
{\rm var}\bigl(f,\Pi_{2}\bigr)\leq(\varrho^{-1}-1){\rm var}_{\pi}(f)+\varrho^{-1}{\rm var}\left(f,\Pi_{1}\right)\quad.
\]
\end{lem}
\begin{proof}
The first result is straightforward. For the second result, first
notice that
\begin{align*}
\sup_{g\in L_{0}^{2}\bigl(\mathsf{E},\mu\bigr)}2\bigl\langle f,g\bigr\rangle_{\mu}-\mathcal{E}_{\Pi_{2}}(g) & \leq\sup_{g\in L_{0}^{2}\bigl(\mathsf{E},\mu\bigr)}2\bigl\langle f,g\bigr\rangle_{\mu}-\varrho\mathcal{E}_{\Pi_{1}}(g)\\
 & =\varrho^{-1}\Big(\sup_{g\in L_{0}^{2}\bigl(\mathsf{E},\mu\bigr)}2\bigl\langle f,\varrho g\bigr\rangle_{\mu}-\mathcal{E}_{\Pi_{1}}(\varrho g)\Big)\\
 & =\varrho^{-1}\Big(\sup_{g\in L_{0}^{2}\bigl(\mathsf{E},\mu\bigr)}2\bigl\langle f,g\bigr\rangle_{\mu}-\mathcal{E}_{\Pi_{1}}(g)\Big)\quad,
\end{align*}
and since ${\rm var}\bigl(f,\Pi\bigr)=2\bigl[\sup_{g\in L_{0}^{2}\bigl(\mathsf{E},\mu\bigr)}2\bigl\langle f,g\bigr\rangle_{\mu}-\mathcal{E}_{\Pi}(g)\bigr]-\|f\|_{\mu}^{2}$
we conclude that
\[
{\rm var}\bigl(f,\Pi_{2}\bigr)\leq(\varrho^{-1}-1){\rm var}_{\mu}(f)+\varrho^{-1}{\rm var}\left(f,\Pi_{1}\right)\quad.
\]

\end{proof}

\section{Supplementary material for Section~\ref{sec:Estimates-of-the} \label{sec:Supplementary-quantitative}}

The proof of Proposition~\ref{prop:doubleconditionalSMC} relies
on the following technical lemma, and is given after this intermediate
result.
\begin{lem}
\label{lem:backwardinduction}Let $x,y\in\mathsf{X}$, then,
\begin{enumerate}[label=(\alph*)]
\item \label{enu:condexpect_a}\textup{for any $t\geq2$, $z_{1:t-1}\in\mathsf{Z}^{N(t-1)}$
such that $(z_{1:t-\text{1}}^{1},z_{1:t-\text{1}}^{2})=(x_{1:t-1},y_{1:t-1})$
and $f_{t}:\mathsf{Z}\rightarrow\mathbb{R}$ we have 
\begin{multline*}
\mathbb{E}_{\mathbf{1},x,\mathbf{2},y}^{N}\left[\sum_{m=1}^{N}f_{t}(Z_{t}^{m})\,\Bigg|\,Z_{t-1}=z_{t-1}\right]\\
=f_{t}(x_{t})+f_{t}(y_{t})+\frac{N-2}{\sum_{l=1}^{N}G_{t-1}(z_{t-1}^{l})}\sum_{k=1}^{N}Q_{t-1,t}(f_{t})(z_{t-1}^{k})\quad,
\end{multline*}
 }
\item \label{enu:condexpect_b}for any $k=1,\dots,T-1$, any \textup{$z_{1:T-k}\in\mathsf{Z}^{N(T-k)}$
such that $(z_{1:T-k}^{1},z_{1:T-\text{k}}^{2})=(x_{1:T-k},y_{1:T-k})$}
\[
\mathbb{E}_{\mathbf{1},x,\mathbf{2},y}^{N}\left[\prod_{t=T-k+1}^{T}\sum_{j=1}^{N}G_{t}(Z_{t}^{j})\,\Bigg|\,Z_{T-k}=z_{T-k}\right]=A_{T-k}+B_{T-k}
\]
where 
\begin{align*}
A_{T-k}: & =\sum_{s=1}^{k}(N-2)^{k-s}\sum_{\mathbf{i}\in\mathcal{I}_{k,s}}\left[G_{T-k+1,i_{\text{1}}}\left(x_{T-k+1}\right)+G_{T-k+1,i_{\text{1}}}(y_{T-k+1})\right]C_{k,s}\bigl(\mathbf{i},x,y\bigr)\quad,\\
B_{T-k}: & =\frac{N-2}{\sum_{l=1}^{N}G_{T-k}(z_{T-k}^{l})}\bigg(\sum_{s=1}^{k}(N-2)^{k-s}\sum_{\mathbf{i}\in\mathcal{I}_{k,s}}\sum_{r=1}^{N}G_{T-k,i_{\text{1}}}\bigl(z_{T-k}^{r}\bigr)C_{k,s}\bigl(\mathbf{i},x,y\bigr)\bigg)\quad,
\end{align*}
and $\mathcal{I}_{k,s}$ and $C_{k,s}$ are as in Proposition~\ref{prop:doubleconditionalSMC}.
\end{enumerate}
\end{lem}
\begin{proof}[Proof of Lemma~\ref{lem:backwardinduction}]
The property in~\ref{enu:condexpect_a} is immediate from the linearity
of the expectation and the definition of the process. We now prove
property~\ref{enu:condexpect_b} by induction on $k=1,\ldots,T-1$.
In order to alleviate notation we let $G_{p,q}^{i}:=G_{p,q}\bigl(Z_{p}^{i}\bigr)$
when found inside an expectation and $G_{p,q}^{i}:=G_{p,q}\bigl(z_{p}^{i}\bigr)$
otherwise, $G_{p,q}^{1+2}:=G_{p,q}\bigl(x_{p}\bigr)+G_{p,q}\bigl(y_{p}\bigr)$
and $C_{k,s}(\mathbf{i}):=C_{k,s}(\mathbf{i},x,y)$. The case $k=1$
follows from~\ref{enu:condexpect_a} with $t=T$ by observing that
$\mathcal{I}_{1,1}=\bigl\{ T+1\bigr\}$, $C_{1,1}\bigl(\mathbf{i},x,y\bigr)=1$
and that $G_{T-1,T+1}^{r}=Q_{T-1,T}\bigl(G_{T}\bigr)\bigl(z_{T-1}^{r}\bigr)$
: 
\begin{align*}
\mathbb{E}_{\mathbf{1},x,\mathbf{2},y}^{N} & \left[\sum_{m=1}^{N}G_{T}^{m}\,\Bigg|\,Z_{T-1}=z_{T-1}\right]\\
 & =G_{T}(x_{T})+G_{T}(y_{T})+\frac{N-2}{\sum_{l=1}^{N}G_{T-1}(z_{T-1}^{l})}\sum_{k=1}^{N}Q_{T-1,T}(G_{T})(z_{T-1}^{k})\\
 & =G_{T}(x_{T})+G_{T}(y_{T})+\frac{N-2}{\sum_{l=1}^{N}G_{T-1}(z_{T-1}^{l})}\sum_{k=1}^{N}G_{T-1,T+1}^{k}\quad.
\end{align*}
Now we assume the property true for some $k\in\{1,\ldots,T-2\}$ and
establish it for $k+1$. We have
\begin{align*}
\mathbb{E}_{\mathbf{1},x,\mathbf{2},y}^{N}\left[\prod_{t=T-k}^{T}\sum_{j=1}^{N}G_{t}^{j}\,\Bigg|\,Z_{T-k-1}=z_{T-k-1}\right] & =A+B\quad,
\end{align*}
with
\begin{align*}
A:= & \mathbb{E}_{\mathbf{1},x,\mathbf{2},y}^{N}\left[A_{T-k}\sum_{j=1}^{N}G_{T-k}^{j}\,\Bigg|\,Z_{T-k-1}=z_{T-k-1}\right]\\
B:= & \mathbb{E}_{\mathbf{1},x,\mathbf{2},y}^{N}\left[B_{T-k}\sum_{j=1}^{N}G_{T-k}^{j}\,\Bigg|\,Z_{T-k-1}=z_{T-k-1}\right]\quad,
\end{align*}
and we deal with the two terms separately. Observe that $A_{T-k}$
only depends on $x_{T-k+1:T}$ and $y_{T-k+1:T}$, then by application
of the first result of the lemma we obtain
\[
A=A_{T-k}\,\left(G_{T-k}^{1+2}+\frac{N-2}{\sum_{l=1}^{N}G_{T-k-1}^{l}}\sum_{l=1}^{N}G_{T-k-1,T-k+1}^{l}\right)
\]
and, noting that $C_{k,s}(\mathbf{i})$ depends on $x_{T-k+2:T},y_{T-k+2:T}$
only 
\begin{align*}
B & =\sum_{s=1}^{k}(N-2)^{k+1-s}\sum_{\mathcal{I}_{k,s}}\mathbb{E}_{\mathbf{1},x,\mathbf{2},y}^{N}\left[\sum_{r=1}^{N}G_{T-k,i_{\text{1}}}^{r}\,\Bigg|\,Z_{T-k-1}=z_{T-k-1}\right]C_{k,s}\bigl(\mathbf{i}\bigr)\\
 & =\sum_{s=1}^{k}(N-2)^{k+1-s}\sum_{\mathcal{I}_{k,s}}\bigg[G_{T-k,i_{\text{1}}}^{1+2}+\frac{N-2}{\sum_{l=1}^{N}G_{T-k-1}^{l}}\sum_{r=1}^{N}G_{T-k-1,i_{\text{1}}}^{r}\bigg]C_{k,s}\bigl(\mathbf{i}\bigr)\quad,
\end{align*}
where we have again applied the first result of the lemma. Consequently
we can group the terms as follows
\begin{multline}
A+B=A_{T-k}G_{T-k}^{1+2}+\sum_{s=1}^{k}(N-2)^{k+1-s}\sum_{\mathcal{I}_{k,s}}G_{T-k,i_{\text{1}}}^{1+2}C_{k,s}\bigl(\mathbf{i}\bigr)\\
+\frac{N-2}{\sum_{l=1}^{N}G_{T-k-1}^{l}}\left[A_{T-k}\sum_{l=1}^{N}G_{T-k-1,T-k+1}^{l}+\sum_{s=1}^{k}(N-2)^{k+1-s}\sum_{\mathcal{I}_{k,s}}\bigg(\sum_{r=1}^{N}G_{T-k-1,i_{\text{1}}}^{r}\bigg)C_{k,s}\bigl(\mathbf{i}\bigr)\right]\quad.\label{eq:A+B}
\end{multline}
Now we first focus on the first term on on the RHS on the first line
(with the sum now written in extension in order to help and we note
that we do not use the double indexing $i_{j}^{s}$ in order to keep
notation simple),
\begin{align*}
A_{T-k}G_{T-k}^{1+2} & =\sum_{s=1}^{k}(N-2)^{k-s}G_{T-k}^{1+2}\sum_{T-k+1<i_{1}\cdots<i_{s-1}<i_{s}=T+1}G_{T-k+1,i_{\text{1}}}^{1+2}\prod_{m=1}^{s-1}G_{i_{m},i_{m+1}}^{1+2}\\
 & =\sum_{s=1}^{k}(N-2)^{k-s}\sum_{i_{0}=T-k+1<i_{1}\cdots<i_{s-1}<i_{s}=T+1}G_{T-k,i_{\text{0}}}^{1+2}\prod_{m=0}^{s-1}G_{i_{m},i_{m+1}}^{1+2}\\
 & =\sum_{s'=2}^{k+1}(N-2)^{k+1-s'}\sum_{j_{1}=T-k+1<j_{2}\cdots<j_{s'-1}<j_{s'}=T+1}G_{T-k,j_{1}}^{1+2}\prod_{m=1}^{s'-1}G_{j_{m},j_{m+1}}^{1+2}\quad,
\end{align*}
where we have used the following changes of variables: $j_{m}=i_{m-1}$
for $m=1,\ldots,s+1$ followed by $s=s'-1$. Note that we can extend
the sum in order to include the term $s'=1$, since we cannot have
$j_{1}=T+1\neq T-k+1=j_{1}$. We examine the second term on the RHS
of the first line of (\ref{eq:A+B}) 
\[
\sum_{s=1}^{k}(N-2)^{k+1-s}\sum_{T-k+1<i_{1}<\cdots<i_{s-1}<i_{s}=T+1}G_{T-k,i_{\text{1}}}^{1+2}\prod_{m=1}^{s-1}G_{i_{m},i_{m+1}}^{1+2}\quad,
\]
and we notice that we can extend the sum in order to include the term
$s=k+1$ because $\sharp\left\{ T-k+2,\ldots,T+1\right\} =k$, which
implies that $\mathcal{I}_{k,k+1}=\emptyset$. Consequently we deduce
that 
\[
A_{T-k}G_{T-k}^{1+2}+\sum_{s=1}^{k}(N-2)^{k+1-s}\sum_{\mathcal{I}_{k,s}}G_{T-k,i_{\text{1}}}^{1+2}C_{k,s}\bigl(\mathbf{i}\bigr)=A_{T-(k+1)}\quad.
\]
We now turn to the second line of (\ref{eq:A+B}) and examine the
two terms within the brackets and use similar ideas. First we have
\begin{align*}
\sum_{l=1}^{N} & G_{T-k-1,T-k+1}^{l}A_{T-k}\\
 & =\sum_{s=1}^{k}(N-2)^{k-s}\sum_{l=1}^{N}G_{T-k-1,T-k+1}^{l}\sum_{T-k+1<i_{1}\cdots<i_{s-1}<i_{s}=T+1}G_{T-k+1,i_{\text{1}}}^{1+2}\prod_{m=1}^{s-1}G_{i_{m},i_{m+1}}^{1+2}\\
 & =\sum_{s=1}^{k}(N-2)^{k-s}\sum_{i_{0}=T-k+1<i_{1}\cdots<i_{s-1}<i_{s}=T+1}\sum_{l=1}^{N}G_{T-k-1,i_{0}}^{l}\prod_{m=0}^{s-1}G_{i_{m},i_{m+1}}^{1+2}\\
 & =\sum_{s=2}^{k+1}(N-2)^{k+1-s}\sum_{i_{1}=T-k+1<i_{2}\cdots<i_{s-1}<i_{s}=T+1}\sum_{l=1}^{N}G_{T-k-1,i_{1}}^{l}\prod_{m=1}^{s-1}G_{i_{m},i_{m+1}}^{1+2}
\end{align*}
and the other term is, in extension,
\[
\sum_{s=1}^{k}(N-2)^{k+1-s}\sum_{T-k+1<i_{1}\cdots<i_{s-1}<i_{s}=T+1}\bigl[\sum_{r=1}^{N}G_{T-k-1,i_{\text{1}}}^{r}\bigr]\prod_{m=1}^{s-1}G_{i_{m},i_{m+1}}^{1+2}\quad.
\]
We therefore conclude that
\begin{align*}
\frac{N-2}{\sum_{l=1}^{N}G_{T-k-1}^{l}}\left[A_{T-k}\sum_{l=1}^{N}G_{T-k-1,T-k+1}^{l}+\sum_{s=1}^{k}(N-2)^{k+1-s}\sum_{\mathcal{I}_{k,s}}\bigl[\sum_{r=1}^{N}G_{T-k-1,i_{\text{1}}}^{r}\bigr]C_{k,s}\bigl(\mathbf{i}\bigr)\right]\\
=B_{T-(k+1)}\;,
\end{align*}
which finishes the proof.
\end{proof}

\begin{proof}[Proof of Proposition~\ref{prop:doubleconditionalSMC}]
We start with the second result of Lemma~\ref{lem:backwardinduction}
for $k=T-1$ and we proceed as in the beginning of the proof of that
lemma, using similar notation and arguments. Here we have however
\begin{align*}
A & =A_{1}\times\left(G_{1}^{1+2}+(N-2)G_{0,2}\right)\\
 & =\sum_{s=2}^{T}(N-2)^{T-s}\sum_{i_{1}=2<i_{2}\cdots<i_{s-1}<i_{s}=T+1}G_{1,i_{1}}^{1+2}\prod_{m=1}^{s-1}G_{i_{m},i_{m+1}}^{1+2}\\
 & +(N-2)\sum_{s=2}^{T}(N-2)^{T-s}\sum_{i_{1}=2<i_{2}<\cdots<i_{s-1}<i_{s}=T+1}G_{0,i_{\text{1}}}\prod_{m=1}^{s-1}G_{i_{m},i_{m+1}}^{1+2}\quad,
\end{align*}
and 
\begin{align*}
B & =\sum_{s=1}^{T-1}(N-2)^{T-s}\sum_{\mathcal{I}_{T-1,s}}\bigl[G_{1,i_{\text{1}}}^{1+2}+(N-2)G_{0,i_{\text{1}}}\bigr]\prod_{m=1}^{s-1}G_{i_{m},i_{m+1}}^{1+2}\\
 & =\sum_{s=1}^{T-1}(N-2)^{T-s}\sum_{2<i_{1}<\cdots<i_{s-1}<i_{s}=T+1}G_{1,i_{\text{1}}}^{1+2}\prod_{m=1}^{s-1}G_{i_{m},i_{m+1}}^{1+2}\\
 & +\sum_{s=1}^{T-1}(N-2)^{T+1-s}\sum_{2<i_{1}<\cdots<i_{s-1}<i_{s}=T+1}G_{0,i_{\text{1}}}\prod_{m=1}^{s-1}G_{i_{m},i_{m+1}}^{1+2}\quad,
\end{align*}
and using arguments similar to those of the proof of Lemma~\ref{lem:backwardinduction},
\begin{align*}
A+B & =\sum_{s=1}^{T}(N-2)^{T-s}\sum_{1<i_{1}<\cdots<i_{s-1}<i_{s}=T+1}G_{1,i_{\text{1}}}^{1+2}\prod_{m=1}^{s-1}G_{i_{m},i_{m+1}}^{1+2}\\
 & +(N-2)\sum_{s=1}^{T}(N-2)^{T-s}\sum_{1<i_{1}<\cdots<i_{s-1}<i_{s}=T+1}G_{0,i_{\text{1}}}\prod_{m=1}^{s-1}G_{i_{m},i_{m+1}}^{1+2}\quad,
\end{align*}
which can be rewritten as (again we use $s'-1=s$ and the fact that
$G_{0,1}=G_{0}=1$ by convention)
\begin{align*}
A+B & =\sum_{s=1}^{T}(N-2)^{T-s}\sum_{i_{0}=1<i_{1}<\cdots<i_{s-1}<i_{s}=T+1}\prod_{m=0}^{s-1}G_{i_{m},i_{m+1}}^{1+2}\\
 & +\sum_{s=1}^{T}(N-2)^{T+1-s}\sum_{1<i_{1}<\cdots<i_{s-1}<i_{s}=T+1}G_{0,i_{\text{1}}}\prod_{m=1}^{s-1}G_{i_{m},i_{m+1}}^{1+2}\\
 & =\sum_{s'=2}^{T+1}(N-2)^{T+1-s'}\sum_{i_{1}=1<i_{2}<\cdots<i_{s'-1}<i_{s'}=T+1}\prod_{m=1}^{s'-1}G_{i_{m},i_{m+1}}^{1+2}\\
 & +\sum_{s=1}^{T}(N-2)^{T+1-s}\sum_{1<i_{1}<\cdots<i_{s-1}<i_{s}=T+1}G_{0,i_{\text{1}}}\prod_{m=1}^{s-1}G_{i_{m},i_{m+1}}^{1+2}\\
 & =\sum_{s=1}^{T+1}(N-2)^{T+1-s}\sum_{0<i_{1}<\cdots<i_{s-1}<i_{s}=T+1}G_{0,i_{\text{1}}}\prod_{m=1}^{s-1}G_{i_{m},i_{m+1}}^{1+2}\quad.
\end{align*}
We conclude.
\end{proof}

\begin{proof}[Proof of Lemma~\ref{lem:A2impliesA1}]
Consider first the case where $k\geq m$, for $z_{p},z'_{p}\in\mathsf{Z}^{2}$,
\begin{align*}
\frac{Q_{p,p+k}(1)(z_{p})}{Q_{p,p+k}(1)(z'_{p})} & =\frac{G_{p}(z_{p})}{G_{p}(z'_{p})}\frac{M_{p+1}\left(G_{p+1}M_{p+2}Q_{p+2,p+k}(1)\right)(z_{p})}{M_{p+1}\left(G_{p+1}M_{p+2}Q_{p+2,p+k}(1)\right)(z'_{p})}\\
 & \leq\delta^{1/m}\frac{\sup_{z\in\mathsf{Z}}G_{p+1}(z)}{\inf_{z'\in\mathsf{Z}}G_{p+1}(z')}\frac{M_{p,p+2}\left(Q_{p+2,p+k}(1)\right)(z_{p})}{M_{p,p+2}\left(Q_{p+2,p+k}(1)\right)(z'_{p})}\\
 & \leq\delta\frac{M_{p,p+m}\left(Q_{p+m,p+k}(1)\right)(z_{p})}{M_{p,p+m}\left(Q_{p+m,p+k}(1)\right)(z'_{p})}\quad,
\end{align*}
by using (A\ref{hyp:strongmixingpotentialassumptions})\ref{hyp:enu:Gcondition}
and a straightforward induction. Now we can conclude by using (A\ref{hyp:strongmixingpotentialassumptions})\ref{hyp:enu:Mcondition}.
When $k<m$ we simply note that, proceeding as above, for any $z_{p},z'_{p}\in\mathsf{X}^{2}$,
\[
\frac{Q_{p,p+k}(1)(z_{p})}{Q_{p,p+k}(1)(z'_{p})}\leq\delta^{k/m}\leq\delta\leq\beta\delta\quad.
\]

\end{proof}

\section{Supplementary material for Section~\ref{sec:conjectureonboundedness}}
\begin{lem}
\label{lem:tightness}Assume that $\{\mu_{x}\}_{x\in\mathsf{X}}$
is a family of finite measures on $(\mathbb{R}^{d},\mathcal{B}(\mathbb{R}^{d}))$
such that $x\mapsto\mu_{x}(A)$ is a measurable mapping for each $A\in\mathcal{B}(\mathsf{\mathbb{R}}^{d})$,
and suppose that $\xi$ is a probability measure on $(\mathsf{X},\mathcal{B}(\mathsf{X}))$.
For any $\epsilon>0$ there exists a set $A\in\mathcal{B}(\mathsf{X})$
such that $\{\mu_{x}\}_{x\in A}$ is tight and $\xi(A)\ge1-\epsilon$.\end{lem}
\begin{proof}
Denote by $B_{r}$ the closed ball of radius $r$ centred at the origin
and define the sets 
\[
A_{k,r}:=\big\{ x\in\mathsf{X}\,:\,\mu_{x}\big(B_{r}^{\complement}\big)\le k^{-1}\big\},
\]
for $k\in\mathbb{N}$ and $r\in\mathbb{R}_{+}$; observe that $A_{k,r}\in\mathcal{B}(\mathsf{X})$.
Define the finite constants 
\[
r_{\epsilon,k}:=\inf\big\{ r\in\mathbb{R}_{+}\,:\,\xi(A_{k,r})\ge1-\epsilon2^{-k}\big\}.
\]
We may define $A:=\cap_{k\ge1}A_{k,r_{\epsilon,k}}$ which satisfies
$\xi(A^{\complement})\le\sum_{k=1}^{\infty}\epsilon2^{-k}=\epsilon$. 
\end{proof}

\section{\label{sec:Comparison-with-Particle}Detailed comparisons with the
PIMH and PMMH}

In this section we contrast the performance properties of the i-cSMC
(resp. PGibbs sampler), as established in Section~\ref{sec:Estimates-of-the}
(resp. Section~\ref{sec:The-particle-Gibbs}), with those of the
Particle Independent Metropolis--Hastings kernel (PIMH) (resp. particle
Marginal Metropolis--Hastings (PMMH)) also proposed in~\citep{andrieu-doucet-holenstein},
which also aims to (indirectly) sample from $\pi$ as defined in Section~\ref{sec:The-i-CSMC}
(resp. Section~\ref{sec:The-particle-Gibbs}). We use notation similar
to that used in Section~\ref{sec:The-i-CSMC} for the i-cSMC algorithm.
The Markov kernel of the PIMH can be defined for $(a,z)\in\mathsf{W}$
(with an obvious abuse of notation in order to alleviate notation),
$W\in\mathcal{B}\bigr(\mathsf{W}\bigr)$ and $N\geq1$ as 
\[
\check{P}_{N}(a,z;W):=\mathbb{E}^{N}\left[\mathbb{I}\{(A,Z)\in W\}\left\{ 1\wedge\frac{\hat{\gamma}_{T}^{N}(Z)}{\hat{\gamma}_{T}^{N}(z)}\right\} +\left\{ 1-1\wedge\frac{\hat{\gamma}_{T}^{N}(Z)}{\hat{\gamma}_{T}^{N}(z)}\right\} \mathbb{I}\{(a,z)\in W\}\right]\;,
\]
where $\mathbb{E}^{N}$ is the expectation corresponding to the law
$\mathbb{P}^{N}$ of the standard SMC algorithm, defined on $\mathsf{W}\times\mathcal{B}\bigl(\mathsf{W}\bigr)$
via the following conditionals, with $z_{t}\in\mathsf{Z}^{N}$ for
$t\in[T]$, $a_{t}\in[N]^{T-1}$ and $a_{T}\in[N]$, 
\[
\mathbb{P}^{N}\left(Z_{1}\in{\rm d}z_{1}\right):=\prod_{i=1}^{N}M_{1}({\rm d}z_{1}^{i})\quad,
\]
and for $t\in\{2,\ldots,T\}$,
\begin{align*}
\mathbb{P}^{N}\big(Z_{t}\in{\rm d}z_{t},A_{t-1}=a_{t-1} & \mid Z_{1:t-1}=z_{1:t-1},A_{1:t-2}=a_{1:t-2}\big)\\
= & \mathbb{P}^{N}\left(Z_{t}\in{\rm d}z_{t},A_{t-1}=a_{t-1}\mid Z_{t-1}=z_{t-1}\right)\\
:= & \prod_{i=1}^{N}\sum_{k=1}^{N}\frac{G_{t-1}(z_{t-1}^{k})}{\sum_{j=1}^{N}G_{t-1}(z_{t-1}^{j})}\mathbb{I}\left\{ a_{t-1}^{i}=k\right\} M_{t}(z_{t-1}^{k},{\rm d}z_{t}^{i})\quad\text{ and}\\
\mathbb{P}^{N}\left(A_{T}=a_{T}\mid Z_{T}=z_{T}\right)= & \frac{G_{T}(z_{T}^{a_{T}})}{\sum_{j=1}^{N}G_{T}(z_{T}^{j})}\quad,
\end{align*}

$A=A_{1:T}$ and $Z:=Z_{1:T}$. We note that $\hat{\gamma}_{T}^{N}(z)$
is not a random quantity in the definition of $\check{P}_{N}(a,z;W)$.
The invariant distribution of the Markov chain, which evolves on $\mathsf{W}$,
is given for any $W\in\mathcal{B}\bigr(\mathsf{W}\bigr)$ by 
\begin{equation}
\check{\pi}^{N}(W):=\sum_{{\bf i}\in[N]^{T}}\frac{1}{N^{T}}\int_{\mathsf{X}}\mathbb{P}_{{\bf i},x}\left((A,Z)\in W\right)\pi({\rm d}x)\quad.\label{eq:def_check_pi}
\end{equation}
As suggested by its name and as shown in~\citep{andrieu-doucet-holenstein},
this algorithm can be interpreted as being a standard independent
Metropolis--Hastings (IMH) kernel with target distribution $\check{\pi}_{N}$
and proposal distribution the standard SMC law $\mathbb{P}^{N}$.
Samples from $\pi$ can be recovered as a byproduct of $A$ and $Z$~\citep{andrieu-doucet-holenstein}
: this should not be surprising since $\check{\pi}$ is the invariant
distribution of the i-cSMC algorithm as seen as a Markov chain on
the extended space $\mathsf{W}$ and not $\mathsf{X}$ solely. The
interpretation as an IMH algorithm allows us to use a well known result
by~\citep{mengersen-tweedie} to deduce that the PIMH is (uniformly)
geometrically ergodic if and only if $\check{\pi}-{\rm ess}\sup_{z}\hat{\gamma}_{T}^{N}(z)<\infty$
with rate $r\leq1-\check{\epsilon}_{N}$ where
\[
\check{\epsilon}_{N}:=\frac{\gamma_{T}}{\check{\pi}_{N}-{\rm ess}\sup_{z}\hat{\gamma}_{T}^{N}(z)}\quad.
\]
Clearly $\check{\epsilon}_{N}>0$ whenever $\check{\pi}_{N}-{\rm ess}\sup_{z}\hat{\gamma}_{T}^{N}(z)<\infty$,
which is similar to what we have obtained in Propositions~\ref{prop:boundepsilonNwithuniformboundG}
and~\ref{prop:mixing_bound} for the i-cSMC. An important difference,
which may explain the widely perceived superiority of the i-cSMC,
is that the rate of convergence of PIMH will typically not improve
(and in particular converge to $1$) as $N$ increases, even for bounded
potentials, which is in contrast with the corresponding convergence
rate of the i-cSMC (see Propositions~\ref{prop:boundepsilonNwithuniformboundG}
and~\ref{prop:mixing_bound}). 

We can also compare the results of Section~\ref{sec:The-particle-Gibbs}
for the PGibbs sampler with the corresponding results for the PMMH
algorithm \citep{andrieu-doucet-holenstein}. This latter algorithm
evolves on $\Theta\times\mathsf{W}$ with transition probability
\begin{align*}
\check{\Phi}_{N}\bigl(\theta,a,z;S\times W\bigr)= & \int_{S}\mathbb{E}_{\vartheta}^{N}\left[\mathbb{I}\{(A,Z)\in W\}\left\{ 1\wedge\frac{\varpi({\rm d}\vartheta)q(\vartheta,{\rm d}\theta)}{\varpi({\rm d}\theta)q(\theta,{\rm d}\vartheta)}\frac{\hat{\gamma}_{\vartheta,T}^{N}(Z)}{\hat{\gamma}_{\theta,T}^{N}(x)}\right\} \right]q(\theta,{\rm d}\vartheta)\\
+\int_{\Theta}\mathbb{E}_{\vartheta}^{N} & \left[\left\{ 1-1\wedge\frac{\varpi({\rm d}\vartheta)q(\vartheta,{\rm d}\theta)}{\varpi({\rm d}\theta)q(\theta,{\rm d}\vartheta)}\frac{\hat{\gamma}_{\vartheta,T}^{N}(Z)}{\hat{\gamma}_{\theta,T}^{N}(x)}\right\} \mathbb{I}\left\{ (\theta,a,z)\in S\times W\right\} \right]q(\theta,{\rm d}\vartheta)
\end{align*}
which leaves the distribution $\pi({\rm d}\theta)\check{\pi}_{\theta}^{N}({\rm d}w)$
invariant, where for any $\theta\in\Theta$, $\check{\pi}_{\theta}^{N}$
is as in (\ref{eq:def_check_pi}) but with $\mathbb{P}_{\theta}^{N}$
(and $\mathbb{E}_{\theta}^{N}$) corresponding to the SMC process
as defined above for a family of Markov kernels $\{M_{\theta}\}$
and potentials $\{G_{\theta,t},t\in[T]\}$. Just as $\Phi_{N}$ can
be viewed as an exact approximation of $\Gamma$, $\check{\Phi}_{N}$
can be viewed as an exact approximation of a Markov kernel $\Phi^{*}$,
evolving only on $\Theta$ as, for $(\theta,S)\in\Theta\times\mathcal{B}\bigl(\Theta\bigr)$
\begin{align*}
\Phi^{*}\bigl(\theta,S\bigr)= & \int_{S}\left\{ 1\wedge\frac{\varpi({\rm d}\vartheta)q(\vartheta,{\rm d}\theta)}{\varpi({\rm d}\theta)q(\theta,{\rm d}\vartheta)}\frac{\gamma_{\vartheta,T}}{\gamma_{\theta,T}}\right\} q(\theta,{\rm d}\vartheta)\\
 & +\int_{\Theta}\left\{ 1-1\wedge\frac{\varpi({\rm d}\vartheta)q(\vartheta,{\rm d}\theta)}{\varpi({\rm d}\theta)q(\theta,{\rm d}\vartheta)}\frac{\gamma_{\vartheta,T}}{\gamma_{\theta,T}}\right\} \mathbb{I}\left\{ \theta\in S\right\} q(\theta,{\rm d}\vartheta)\;.
\end{align*}
In~\citep{andrieu2015}, it is shown that when
\[
\pi-{\rm ess}\sup_{\theta}\left(\check{\pi}_{\theta}-{\rm ess}\sup_{z}\frac{\hat{\gamma}_{\theta,T}^{N}(z)}{\gamma_{\theta,T}}\right)<\infty\quad,
\]
${\rm Gap}(\check{\Phi}_{N})>0$ whenever ${\rm Gap}(\Phi^{*})>0$,
i.e. the existence of a spectral gap of $\Phi^{*}$ is ``inherited''
by $\check{\Phi}_{N}$. This coincides in many cases with inheritance
of geometric ergodicity, for example when $\check{\Phi}_{N}$ is positive.

The rate of convergence of a geometrically ergodic PMMH Markov chain
does not improve in general as $N$ increases, in contrast to our
results for PGibbs Markov chains. In this context, weak convergence
in $N$ of the asymptotic variance of estimates of $\pi(f)$ using
$\check{\Phi}_{N}$ to that of $\Phi^{*}$ is nevertheless provided
by~\citep[Proposition~19]{andrieu2015} for all $f\in L^{2}(\Theta,\pi)$.
This can be contrasted with quantitative bounds obtained in Theorem~\ref{thm:MWG-functions-of-theta}.

\bibliographystyle{abbrv}
\bibliography{backwardrecursion_arxiv}

\end{document}